\documentclass[12pt, reqno]{amsart}

\usepackage{amssymb}
\usepackage{amsmath}
\usepackage{graphicx}
\usepackage{xcolor}
\usepackage[all]{xy}
\usepackage{enumitem}

\newtheorem{theorem}{Theorem}[subsection]

\newtheorem{lemma}[theorem]{Lemma}
\newtheorem{corollary}[theorem]{Corollary}
\theoremstyle{definition}
\newtheorem{definition}[theorem]{Definition}
\newtheorem{example}[theorem]{Example}
\theoremstyle{remark}
\newtheorem{remark}[theorem]{Remark}
\numberwithin{equation}{section}

\DeclareMathOperator{\Sing}{Sing}
\newcommand*{\nn}{\mathbb N}
\newcommand*{\rr}{\mathbb R}
\newcommand*{\zz}{\mathbb Z}
\newcommand*{\cc}{\mathbb C}
\newcommand\prenonwandering{pre{-}nonwandering}
\newcommand\preperiodic{pre{-}periodic}

\newcommand\still{still }

\newcommand\forever{forever }
\newcommand\classicwanderingset{W}
\newcommand\classicnonwanderingset{\Omega}
\newcommand\stillwanderingset{\overline{W}}
\newcommand\nonwanderingset{\overline{\Omega}}
\newcommand\foreverwanderingset{\widetilde{W}}
\newcommand\prenonwanderingset{\widetilde{\Omega}}

\begin{document}

\title[Pseudo-B\"{o}ttcher components of the wandering set]%
{Pseudo-B\"{o}ttcher components of the wandering set of inner mappings}
\author{Igor \,Yu. Vlasenko}
\address{Algebra and Topology Department, Institute of Mathematics, Kyiv, Ukraine}
\email{vlasenko@imath.kiev.ua}

\begin{abstract}
  This article explores the topology of Pseudo-Böttcher totally
  invariant connected components of the wandering set in dynamical
  systems generated by non\discretionary{-}{}{-}invertible inner (open
  surjective isolated) mappings of compact surfaces. We describe the
  possible topological types of these invariant connected subsets,
  which 
  are more diverse then corresponding components of homeomorphisms.
\end{abstract}

\maketitle


\section{Introduction}
This article contributes to the understanding of the wandering set in
dynamical systems generated by non\discretionary{-}{}{-}invertible
inner (open surjective isolated) mappings of compact surfaces.
It focuses on invariant connected subsets (components) of the wandering set
which can be distinguished using filtration of the wandering set
by attractor\discretionary{-}{}{-}repeller pairs.
Such components are well understood in the invertible case
but has not been studied in the non\discretionary{-}{}{-}invertible case.

In the invertible case their topological type can only be a disc or an annulus.
In the non\discretionary{-}{}{-}invertible case we introduce
a class of annulus\discretionary{-}{}{-}like totally invariant connected components
of the wandering set of dynamical systems,
generated by inner mappings of compact surfaces,
called Pseudo-B\"{o}ttcher components (Definition~\ref{def:pseudobottcher_basin1}).
We describe all possible topological types of
Pseudo-B\"{o}ttcher components, including new topological types
specific to non\discretionary{-}{}{-}invertible case, in Theorem~\ref{th:component_topology}.

Dynamical systems generated by
non\discretionary{-}{}{-}invertible maps are not mainstream,
so there is a need to explain maps, invariant sets, classes and
topological invariants used
and their difference from mainstream invertible dynamical systems in detail.

Inner mappings as a separate subclass of non\discretionary{-}{}{-}invertible continuous
maps of two-dimensional manifolds were 
introduced by S.~Stoilov~\cite{MR0082545_Stoilov56}.
Stoilov defined inner mappings as open zero\discretionary{-}{}{-}dimensional
continuous
maps, proved that inner mappings are in fact
open discrete maps, and proved his well-known theorem
that the class of inner mappings coincides with the class of perturbations
of holomorphic mappings by the right action of the group of
homeomorphisms.
In the
case of higher\discretionary{-}{}{-}dimensional manifolds,
the class of open zero-dimensional maps no longer coincides with the class of open discrete maps.
That is why Yu.~Yu.~Trokhymchuk~\cite{Trokhimchuk2008__en} defined the class of inner mappings
in general case as a class of open discrete maps (Definition~\ref{def:inner_mapping}).
Homeomorphisms, diffeomorphisms, holomorphic maps 
belong to the class of inner mappings.

Classical topological theory of dynamical systems generated by
invertible maps 
considers them up to conjugacy with a homeomorphism
and defines a lot of well-known invariant sets such as wandering and
nonwandering ones.
It is usually hard to apply 
those classic invariant sets to non\discretionary{-}{}{-}invertible maps.
At most,
``$\omega$'' wariants of invariant sets are used, such as
$\omega$-\hspace{0em}nonwandering set.
However, among all continuous maps, inner mappings are
the natural class to which the methods of topological theory of invertible dynamical systems
can be extended.
Being open, inner mappings preserve base topology at a point,
and being discrete, they allow to trace locally a partial
trajectory of a point back-and-forth unambiguously.

Analogues of classic invariant sets for inner mappings
are built in author's book~\cite{VlaBook2014__en}.
We can speak about trajectories (Definition~\ref{def:trajectory}),
wandering sets (Definition~\ref{def:wandering}) and so on,
because the corresponding definitions reduce to classic ones when
applied to invertible maps.

The wandering set of non\discretionary{-}{}{-}invertible
inner mappings can be just as complex as the nonwandering set
of invertible mappings.
In this paper, we focus on a specific subsets of the wandering set:
totally invariant subsets (Definition~\ref{def:totally_invariant}).
A new class called
Pseudo\discretionary{-}{}{-}B\"{o}ttcher regular components is defined
(Definition~\ref{def:pseudobottcher_basin1})
which extends the class
of B\"{o}ttcher components from~\cite{Vla:BRD4:UMG}.
This paper answers the questions about
topology of Pseudo\discretionary{-}{}{-}B\"{o}ttcher regular components.
Similar questions were left unexplored in the paper~\cite{Vla:BRD4:UMG}.
Theorem~\ref{th:component_topology} shows that without singular points
Pseudo\discretionary{-}{}{-}B\"{o}ttcher regular components
are just open annuli that is the same as in the invertible case,
but with singular points those components
become quite complex topologically, as
described in Theorem~\ref{th:component_topology}.

\section{Preliminary information}

Let $f\colon M\to M$ be a continuous surjective map of a compact manifold $M$
without boundary.


Due to
the difference between invertible and non\discretionary{-}{}{-}invertible maps
the notion of invariant set has to be split into two
notions.

\begin{definition}\label{def:invariant_set}\label{def:totally_invariant}
 A set $X\subset M$ is called
  \begin{itemize}
  \item
  \emph{invariant} set 
  if $f(X)=X$.
  \item
  \emph{totally invariant} set 
  if $f^{-1}(X)=X$.
  \end{itemize}
\end{definition}

A totally invariant set is
invariant because $f(X)=f(f^{-1}(X))=X$.
The converse is not true for non\discretionary{-}{}{-}invertible maps.
A set $X$ can be invariant ($f(X)=X$) but not totally invariant
($f^{-1}(X)\not=X$).

For the invertible maps any invariant set is totally invariant.

\subsection{Attractor\discretionary{-}{}{-}repeller pairs.}
\label{sec:attractor_repeller_pair}

Filtration of the manifold $M$ 
by attractor\hspace{0em}{-}repeller pairs
is part of the theory of chain\discretionary{-}{}{-}recurrent sets
and the ``fundamental theorem of dynamical systems''
which has been developed by Conley~\cite{Conley78} for
continuous flows on a compact metric space
and has been generalized by Franks for homeomorphisms~\cite{Franks88}
and Hurley for semiflows and endomorphisms~\cite{Hurley95}.

\begin{definition}\label{def:strictly_attracting_set}
  An open set $U\subset M$ is called \emph{strictly attracting} if $f(\overline{U})\subset U$.
\end{definition}
A strictly attracting open set $U$ generates
\begin{itemize}
  \item
the \emph{attractor}
$A:=\cap_{n\ge 0} f^n(\overline{U})$,
  \item
the \emph{basin of attraction}
$B_A:=\cup_{n\ge 0} f^{-n}(U)$ and
  \item
the \emph{repeller} $R:=M\setminus B_A$.
\end{itemize}

\begin{figure}[htb]
    \includegraphics[scale=0.4]{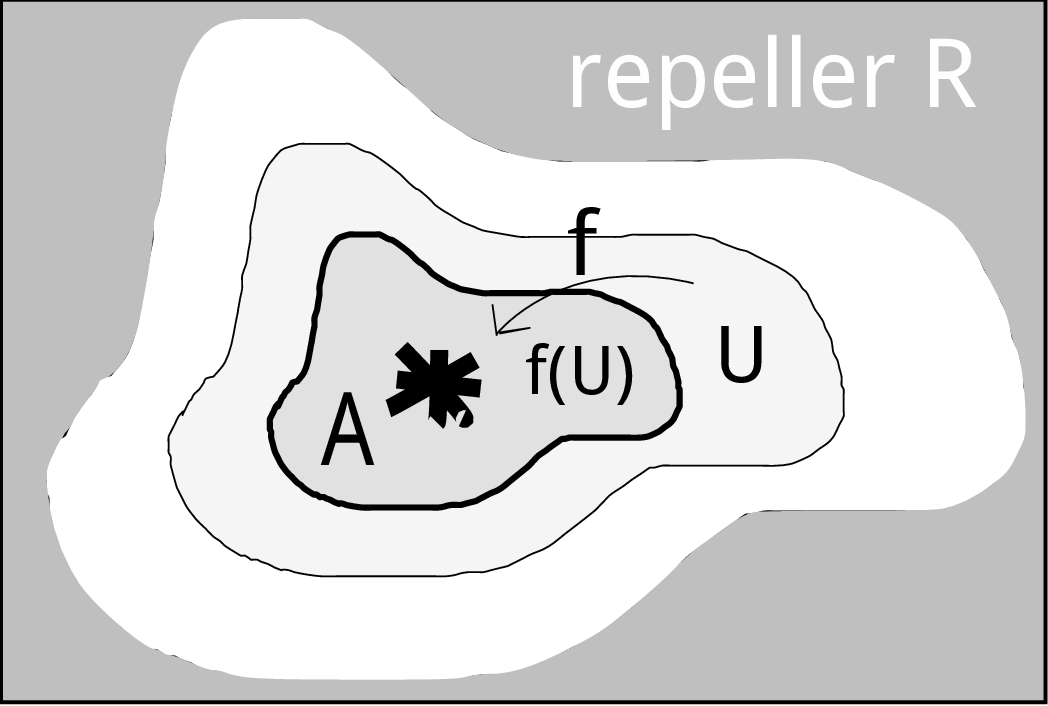}
\caption{
Attractor\discretionary{-}{}{-}repeller pair $(A, R)$ and a strictly attracting
neighborhood $U$.
}
\end{figure}

$A\not=\emptyset$ and $R\not=\emptyset$ because $M$ is compact.
By construction, $f(A)=A$, $f(B_A)=B_A$, $f^{-1}(B_A)=B_A$.
Therefore, $f(R)=R$, $f^{-1}(R)=R$.
If $f$ is not invertible, $f^{-1}(A)$ can be different from $A$.

Repeller $R$ and basin of attraction $B_A$ are 
totally invariant sets. Attractor $A$
is 
invariant, but possibly not totally invariant set.

For homeomorphisms, an attractor\discretionary{-}{}{-}repeller pair $(A, R)$
is symmetric: an attractor $A$ for the map $f$
becomes repeller for the map $f^{-1}$.
For non\discretionary{-}{}{-}invertible maps
an attractor\discretionary{-}{}{-}repeller pair $(A, R)$ is not symmetric:
repeller $R$ is totally invariant, while attractor $A$ is not.

\begin{definition}\label{def:region_of_strict_wandering}
  A set $D:=B_A \setminus \cup_{n\ge 0} f^{-n}(A)$ is called
  the \emph{region of strict wandering} of the attractor $A$.
\end{definition}

\begin{definition}\label{def:component_of_strict_wandering}
  A connected component $S$ of the region of strict wandering 
  is called
  a \emph{component of strict wandering} of the attractor $A$.
\end{definition}
By construction, the region of strict wandering and its
components of strict wandering are open sets.


\subsection{Definition of inner mappings.}

\begin{definition}\label{def:inner_mapping}
  A map $f$ is called
  \begin{enumerate}
  \item
    \emph{open} if the image of any open set is open;
  \item
    \emph{closed} if the image of any closed set is closed;
  \item
    \emph{discrete} if the preimage of any point is
  a discrete set (a set is called \emph{discrete} if any point of the set is isolated);
  \item
    \emph{finite-to-one} if the preimage of any point is finite;
  \item
    \emph{inner mapping} if $f$ is
  continuous open discrete surjective map of $M$.
\end{enumerate}
\end{definition}

\begin{theorem}[S.~Stoilov~\cite{MR0082545_Stoilov56}]\label{th:stoilov_theorem}
  Let $M$ be a two-dimensional manifold and $f\colon M\to \cc$ be an
  inner mapping. Then there exists a Riemannian surface $R$ which is a complex
  structure on $M$, a homeomorphism $h\colon M\to R$,
  and a holomorphic function $F\colon R\to C$
  such that
  $f=F\circ h$.
\end{theorem}

\begin{definition}
  A map $f\colon M\to M$ is called \emph{local homeomorphism} at a point $x\in M$ if
  there exists a neighborhood $U$ of $x$ such that restriction
  $f|_{U}$ is a homeomorphism.
\end{definition}
\begin{definition}
  A map $f\colon M\to M$ is called \emph{covering} at a point $x\in M$ if
  there exists an open neighborhood $U$ of the point $f(x)$ such that
  $f^{-1}(U)=\cup_{x_i\in f^{-1}(x)}V_{x_i}$ where
  \begin{itemize}
  \item $V_{x_i}\cap V_{x_j}=\emptyset$ if $x_i\not=x_j$;
  \item
    for each $x_i\in f^{-1}(x)$ $f(V_{x_i})=U$;
  \item
     the restriction $f|_{V_{x_i}}$ is a homeomorphism.
  \end{itemize}
\end{definition}

For the finite-to-one maps the notions of local homeomorphism and a
covering at a point are equivalent.

\begin{definition}
  A map $f\colon M\to M$ is called
  \begin{itemize}[leftmargin=*]
    \item
    \emph{local homeomorphism} if
    $f$ is local homeomorphism at every point $x\in M$;
  \item
    \emph{covering} if
    $f$ is covering at every point $x\in M$;
  \item
    \emph{branched covering} if
    $f$ is covering at every point $x\in M$ except from a nowhere
    dense set known as branch set.
  \end{itemize}
\end{definition}
\begin{definition}
  A point $x\in M$ is called \emph{singular} point of $f$ if $f$ is
  not a local homeomorphism at the point $x$.
\end{definition}
Denote by $\Sing(f)$ the set of singular points of $f$.

\subsection{Some properties of inner mappings.}


Let $f\colon M\to M$ be an inner mapping
of a compact connected two-dimensional manifold $M$ without boundary.

\begin{lemma}[a collection of facts about inner mappings, see~\cite{MR0082545_Stoilov56,Trokhimchuk2008__en,VlaBook2014__en}]\label{lm:collection_inner_mapping_properties}
Let $f\colon M\to M$ be an inner mapping
of a compact connected two-dimensional manifold $M$
with topological defree $\deg f$.
Then
\begin{enumerate}[label={\rm(\arabic*)}]
\item $f$ is finite-to-one map;
\item $f$ is closed map;
\item $\Sing(f)$ is finite;
\item for any point from $M\setminus \Sing(f)$
  the number of preimages is constant
  and is equal to $|\deg f|$;
\item for any point $p \in \Sing(f)$
  the number of preimages $|f^{-1}(p)|$ is less then $|\deg f|$;
\item $f$ is branched covering with the branch set $\Sing(f)$.
\end{enumerate}
\end{lemma}

\begin{definition}
  For a $p\in M$ 
  define its \emph{defect number} $e(p)$: $e(p)=|\deg f| - |f^{-1}(p)|$.
\end{definition}

\begin{lemma}[Stoilov Lemma, S.~Stoilov~\cite{MR0082545_Stoilov56}]\label{lm:stoilov_lemma}
  Let $p\in \Sing(f)$. Then there exists an open neighborhood $U(p)$ and
  a homeomorphism $h\colon U\to U$ such that $h(p)=f(p)$ and
  $f|_{U}=(z-f(p))^{e(p)+1}\circ h$ in some local coordinates on $U$.
\end{lemma}

It follows that points with defect number $e(p)=0$ are the points of local
homeomorphism, whereas points with defect number $e(p)>0$ are the singular
points of $f$.
The local behavior of an inner mapping around a singular point $p$
is determined by complex function $z^{e(p)+1}$. In particular,
singular points are isolated.


\section{Topological dynamics of inner mappings and its difference in case of homeomorphisms}

Let $M$ be a compact connected two dimensional manifold without boundary.
Let $f\colon M\to M$ be an inner mapping, i.e.
a continuous non\discretionary{-}{}{-}invertible open (an image of an open set is open)
finite-to-one surjective map of $M$.
The iterations of $f$ create a dynamical system on $M$.

Topological theory of dynamical systems study
maps up to topological conjugacy.
For a map $f$ its conjugacy class is the set $\{h^{-1}\circ f \circ h| h \in \textit{Homeo}(M) \}$.
However, since topological dynamics methods were developed for
homeomorphisms, they cannot be directly applied to inner mappings. We
will address this challenge by adapting classical topological dynamics
definitions to work with non\discretionary{-}{}{-}invertible inner
mappings.

For more detail, see the author's book~\cite{VlaBook2014__en}.

\subsection{Trajectories.}

\begin{definition}\label{def:trajectory}\label{def:perp_trajectory}
  Consider a point $x\in M$.
  \begin{itemize}
  \item
    $O^+_f(x)=\{f^n(x)|\ n\ge 0\}$ is a
    \emph{positive semitrajectory} of $x$;
  \item
    $O^-_f(x)=\{f^{-n}(x)|\ n \ge 0\}$
    is a
    \emph{negative semitrajectory} of $x$;
  \item
    A sequence $(x_n)_{n\in \zz}$, where $x_0=x$, $f(x_n)=x_{n+1}$
    is called a
    \emph{partial trajectory} of $x$;
  \item
    $O_f(x)=\cup_{y\in O^+_f(x)} O^-_f(y)$
    is a
    \emph{full trajectory} of $x$;
  \item
    $O^\perp_f(x)=\{f^{-n}\left(f^n(x)\right)|\ n\ge 0\}$
    is a
    \emph{neutral section} of (the full trajectory of) $x$;
  \end{itemize}
\end{definition}

\begin{figure}[htbp]
    \includegraphics[scale=0.7]{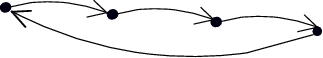}
    \includegraphics[scale=0.8]{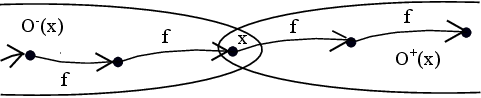}
\caption{
Periodic and generic $O_f(x)$ for homeomorphisms. 
}
\end{figure}

\begin{figure}[htbp]
    \includegraphics[scale=0.5]{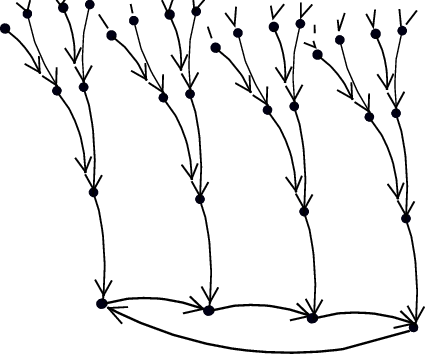}
    \includegraphics[scale=0.66]{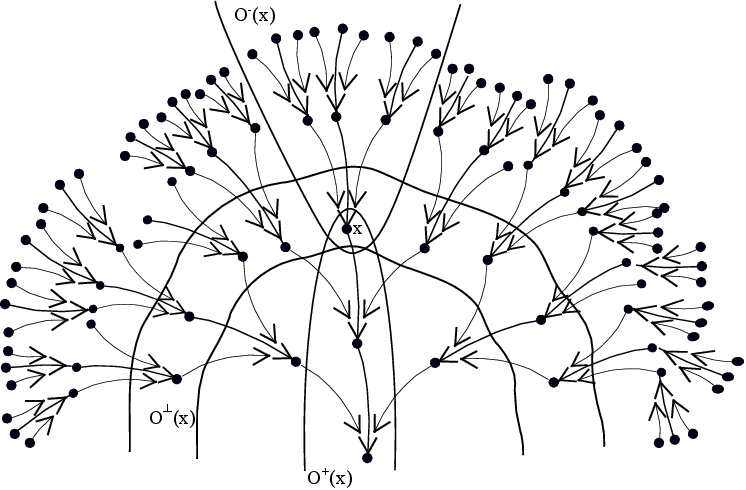}
\caption{
Periodic and generic $O_f(x)$ for inner mappings.
}
\end{figure}


By definition,
$O^+_f(x)\subset O_f(x)$,
$O^-_f(x)\subset O_f(x)$,
$O^\perp_f(x)\subset O_f(x)$.
Note that
partial trajectories of $x$ are not uniquely defined: in the case of
non\discretionary{-}{}{-}invertible maps the point $x$
can have infinitely many different partial trajectories.
Any partial trajectory of $x$ is an invariant set and is contained in $O_f(x)$.
Also notice that $O_f(x)$ is a totally invariant set.

If $f$ is invertible, then $O^\perp_f(x)$ coincide with $\{x\}$, while
$O_f(x)$ and partial trajectories turn into the set 
$\{f^n(x)|\ n\in\zz \}$.

Evidently, the full trajectory,
$O_f(x)$ is the smallest topological totally invariant set preserved
under topological conjugacy.
However, for general non\discretionary{-}{}{-}invertible maps, their full
trajectories are unique sets
that can differ significantly from each other topologically.
This makes describing the topological conjugacy classes of general
non-invertible maps significantly more challenging.
In contrast, for homeomorphisms, there are only a few
possible trajectory types: a generic trajectory without periodic points
and periodic trajectories of period $n$, $n\ge 1$.
This fact explains why classical topological dynamics is well-developed for homeomorphisms but not for general non\discretionary{-}{}{-}invertible maps.

Among non\discretionary{-}{}{-}invertible maps, the class of inner mappings exhibits
dynamical systems most similar to homeomorphisms.
For inner mappings, by definition, the full trajectory, $O_f(x)$, is a closed
and discrete set.  If there are no singular points, then for each
degree $d$, $|d|>1$, for an inner mapping $f$ of degree $d$,
each point has exactly $|d|$ preimages.
Furthermore, the set of all possible trajectories consists only of the
generic trajectory without periodic points and periodic trajectories of period $n$, $n\ge 1$,
which is similar to the case of homeomorphisms.
The presence of a finite number of singular points introduces additional
complexities, but due to their finite number, accounting for all
possible classes of trajectories of inner mappings remains
manageable.

\subsection{Neutrally saturated sets.}

\begin{definition}\label{def:neutrally_saturated}
  A set $X\subset M$ is called
  \emph{neutrally saturated},
  if $\forall x\in X$ $O^\perp(x)\subset X$.
\end{definition}
In particular, a totally invariant set is neutrally saturated.

For any set $X$ the set $X^\perp=\cup_{n\ge 0}f^{-n}\circ f^n(X)$
is a smallest neutrally saturated set such that $X\subset X^\perp$.
\begin{definition}\label{def:neutral_saturation}
  The set $X^\perp$ for a set $X\subset M$ is called the
  \emph{neutral saturation} of the set $X$.
\end{definition}

For example,
consider an attractor $A$ such that $A$ is not totally invariant, i.e. $f^{-1}(A)\not=A$.
Then
the set $A^\perp\setminus A$ consists of additinal preimages of $A$:
$A^\perp=\cup_{n\ge 0}f^{-n}\circ f^n(A)$, $f^n(A)=A$ $\Rightarrow$
$A^\perp=\cup_{n\ge 0}f^{-n}(A)$.

\subsection{Wandering and nonwandering points.}

General non\discretionary{-}{}{-}invertible maps do not preserve neighborhoods of points.
For example, a simple map $x^2$ maps the open neighborhood $(-1,1)$ of $0$
to the semi-interval $[0,1)$ which is neither open set, nor closed set, and not a
neighborhood of $0$ (it does not contain a metric $\varepsilon$-ball for any $\varepsilon>0$).
It only gets worse in dimensions 2 and above.
In contrast, in the case of homeomorphisms 
a neighborhood $U(x)$ of a point $x\in M$ can be mapped to any point
of $O_f(x)$ by $f^{n}$ or $f^{-n}$, and the image of $U(x)$ is, by
definition, homeomorphic.

In this regard inner mappings resemble homeomorphisms.
As inner mappings are open and close maps, the image and preimage of
open or closed neighborhood is also open or closed neighborhood.
Even more, inner mappings have the structure of branched coverings.
This means that if $U(x)$ is small enough
and $O_f(x)$ does not have any singular points of $f$,
then we can build a chain of local homeomorphisms,
moving $U(x)$ to any point of $O_f(x)$ homeomorphically.
If there is a singular point, we can still build an image of $U(x)$ locally
which is not a homeomorphic image, but still
remains open or closed connected neighborhood.

\begin{figure}[htb]
   a)\includegraphics[scale=0.6]{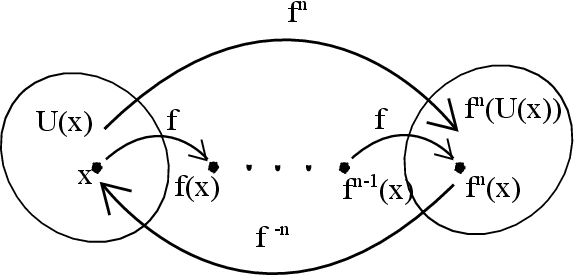}
   b)\includegraphics[scale=0.6]{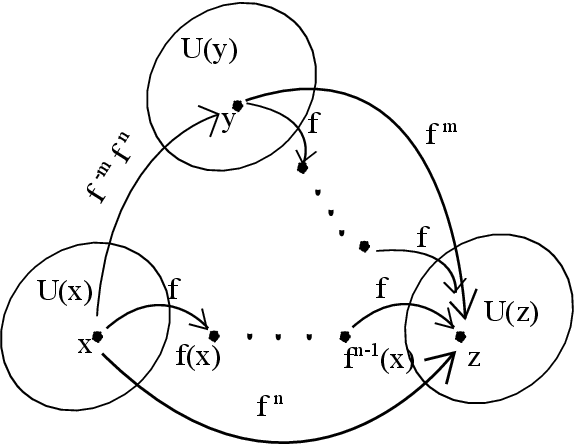}
\caption{
Dynamics of neighborhoods along $O_f(x)$. a) homeomorphisms b) inner mappings.
}
\end{figure}

\begin{definition}\label{def:wandering}
  Denote by $U(x)$ an open neighborhood of a point $x$.
  A point $x$ is called
  \begin{itemize}
  \item
  \emph{nonwandering} point of $f$
    if for any neighborhood $U(x)$ and for all
    $m\in {\mathbb Z}$, $m\not=0$:  $f^{m}(U) \cap U \not= \emptyset$.
  \item
  \emph{\still wandering} point of $f$ if $x$ isn't nonwandering.
  \item
  \emph{\forever wandering} or, for short, \emph{wandering} point of $f$
  if there exists a neighborhood $U(x)$ such that $\forall n\ge 0$,
  $\forall m\in {\mathbb Z}$, $m\not=n$:  $f^{m}\left(f^{n}\left(U\right)\right) \cap U = \emptyset$.
  \item
  \emph{\prenonwandering{}} point of $f$
  if $x$ is not \forever wandering. 
  \item
\emph{\prenonwandering{} wandering} if $x$ is \prenonwandering{} but not nonwandering.
  \end{itemize}
\end{definition}

Note that if $x$ is \prenonwandering{} point of $f$ then from definition there exists
$n\ge 0$ such that $f^{n}(x)$ is a nonwandering point of $f$.

Denote by $\nonwanderingset(f)$, $\prenonwanderingset(f)$, $\stillwanderingset(f)$ and $\foreverwanderingset(f)$ the sets of nonwandering,
\prenonwandering{}, \still wandering and \forever wandering points of $f$.
By definition, $\nonwanderingset(f)\subset\prenonwanderingset(f)$,
$\stillwanderingset(f)\subset\foreverwanderingset(f)$.
The set $\nonwanderingset(f)$ is invariant but not totally invariant.
The set $\stillwanderingset(f)$ is not invariant when $f$ is not
invertible due to existsence of \prenonwandering{} wandering points.
The sets $\prenonwanderingset(f)$ and its counterpart $\foreverwanderingset(f)$
are totally invariant.

When $f$ is a homeomorphism, the sets
$\foreverwanderingset(f)$ and $\prenonwanderingset(f)$ both become 
classic wandering and nonwanderings sets of homeomorphisms
$\classicwanderingset(f)$ and $\classicnonwanderingset(f)$.
The same is true for $\stillwanderingset(f)$ and $\nonwanderingset(f)$,
because the definitions of $\stillwanderingset(f)$ and $\nonwanderingset(f)$
are just verbatim definitions of wandering and nonwanderings sets of homeomorphisms,
applied 
to non\discretionary{-}{}{-}invertible case.
We focus on the set $\foreverwanderingset(f)$
as the non\discretionary{-}{}{-}invertible analog to the classical set
of wandering points, $\classicwanderingset(f)$,
because the alternative choice the set $\stillwanderingset(f)$ is not invariant.

\subsection{Decomposition of wandering set of homeomorphisms.}
\label{sec:decomposition_of_wandering_set}

Essentially, when restricted to their wandering sets, dynamical
systems of homeomorphisms become discrete $\zz$-actions.
In contrast to some simple $\zz$-actions, such as
$z\mapsto z+1$ on $\cc$ or $z\mapsto 2z$ on $\cc\setminus \{0\}$,
a typical wandering set of homeomorphisms as a whole is still topologically intricate and
can be further decomposed into ``elementary blocks'', open pieces with simple topology and
with $\zz$-actions that have uniform $-\infty$ and $+\infty$ directions,
and some ``separatrix sets'', which divide the
wandering set into those ``elementary blocks''.

An example of such decomposition is the classical theory of regular subsets of
the wandering set, introduced by Birkhoff \cite{BS}.
Another example is filtration of the wandering set by attractor\discretionary{-}{}{-}repeller pairs
described in Section~\ref{sec:attractor_repeller_pair}.
We won't go into the details regarding those decompositions, as
they are used purely as motivation of the research.
Let's limit ourselves to the class of ``elementary blocks'' which
consists of invariant open connected sets.
\nocite{BirkgoffDS, BS}

\begin{figure}[htbp]\label{fig:homeomorphisms-on-S}
    a)\includegraphics[scale=0.3]{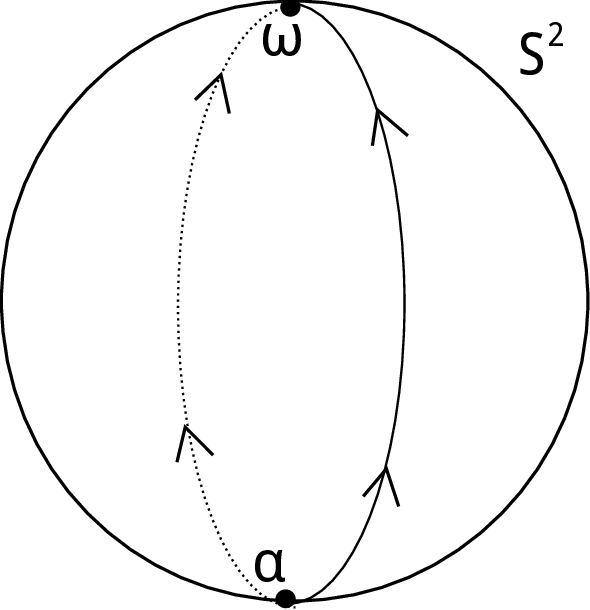}
    b)\includegraphics[scale=0.3]{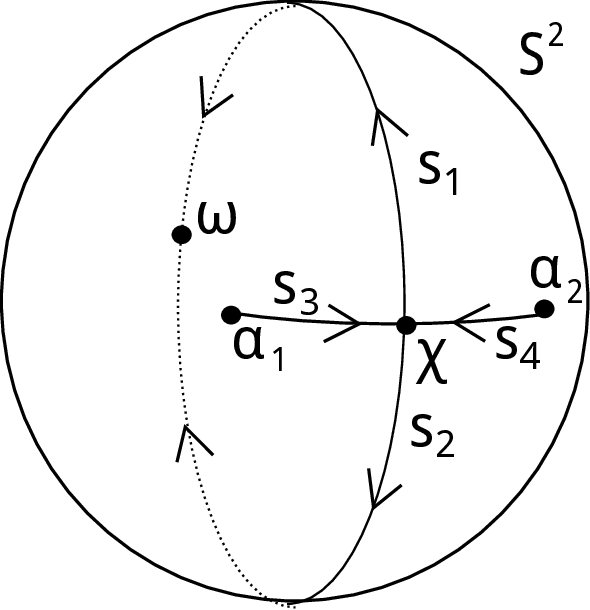}
\caption{
Homeomorphisms on $S^2$ and their wandering sets.
}
\end{figure}

\begin{example}[Figure~\ref{fig:homeomorphisms-on-S}a]\label{ex:regcomp_S2}
Consider a dynamical system on a sphere
where the nonwandering set consists only of two points:
a sink (attracting fixed point) $\omega$ and a source (repelling fixed point) $\alpha$
which create attractor\discretionary{-}{}{-}repeller pair $(\omega, \alpha)$.
The rest of the sphere is wandering set which is homeomorphic to an open annulus.
This wandering set itself is an ``elementary block''.
Any pair of such homeomorphisms is topologically conjugate.
\end{example}

\begin{example}[Figure~\ref{fig:homeomorphisms-on-S}b]
Consider a Morse-Smale diffeomorphism of a sphere $S^2$.
Its nonwandering set consists of four points:
a sink $\omega$, a saddle $\chi$
and two sources $\alpha_1$ and $\alpha_2$.
There are 3 attractor\discretionary{-}{}{-}repeller pairs.
The 1st pair is attractor $\{\omega\}$ and repeller
$\{\alpha_1$, $s_3$, $\chi$, $s_4$, $\alpha_2\}$.
This pair decompose the wandering set into
a ``separatrix set'' of $s_3$ and $s_4$.
and a single
``elementary block'', which lies in the basin of attraction of $\omega$
and called an $\omega$-regular component by Birkhoff.
The 2nd and 3rd pairs are one of the sources $\alpha_1$, $\alpha_2$
as repeller and closure of the basin of repulsion of another source as
attractor.
They decompose the wandering set into
a ``separatrix set'' of $s_1$ and $s_2$.
and two
``elementary blocks'', which lie in the basins of repulsion of
$\alpha_1$ and $\alpha_2$
and called an $\alpha$-regular components by Birkhoff.
\end{example}

In the examples above all the ``elementary blocks'' are
homeomorphic to annuli. 
In fact,
Birkhoff and Smith~\cite{BirkgoffDS, BS} proved that
``elementary blocks'' of homeomorphisms,
which are invariant open connected sets,
 on compact surfaces
can only be homeomorphic either to a disc or an annulus.
Those results also apply to invariant open connected sets
cut from the wandering set by an attractor\discretionary{-}{}{-}repeller pair.

However, the results of Birkhoff and Smith regarding the homeomorphisms do not apply to
non\discretionary{-}{}{-}invertible inner mappings of surfaces.
A simple example is the holomorphic map $z^2+1$ of the sphere $\overline{\cc}$.
Its Julia set is the repeller and a Cantor set. $\infty\in\overline{\cc}$
is the attractor. Their complement is the Fatou set of $z^2+1$, which for
this map coincides with its wandering set.
Note that the wandering set of $z^2+1$ is totally invariant connected
open set and its complement is the
attractor\discretionary{-}{}{-}repeller pair.
Then it consists of single ``elementary block'' analog for inner mappings.
However, its topological type is not a disc or an annulus, but a disc
with a Cantor set punctured.

Let $f\colon M\to M$ be an inner mapping of a compact surface $M$ such
that $f$ has an totally invariant connected
open subset $S$ of its wandering set cut off by a
attractor\discretionary{-}{}{-}repeller pair $(A, R)$.
The question arises:
what are possible topological types of $S$?
Consider a special case first.
Let's there be a neighborhood $U$ of the attractor $A$
such that $U\cap S$ is an annulus.
What are possible topological types of $S$?
These questions are the motivation for this article.

\subsection{Topological conjugacy problem of wandering set.}
\label{sec:topological_conjugacy_problem_of_wandering_set}

The key difference between wandering sets of homeomorphisms and inner
mappings lies in how they relate to topological conjugacy. For
homeomorphisms, topological invariants,
which distinguish conjugacy classes, are concentrated
within the nonwandering set. In contrast, inner mappings exhibit an
abundance of conjugacy classes even when restricted to their wandering
set.

For instance,
the non\discretionary{-}{}{-}invertible analog of the example~\ref{ex:regcomp_S2}
is an $n$-sheeted branched covering on a sphere where
the nonwandering set consists of two branch points,
an attracting singular sink and a repelling singular source, and the rest of
the sphere are \forever wandering 
points.
There is an example of
a countable sequence of such inner mappings which are not pairwise topologically conjugate
\cite{Vla:BRD5:ODESSA1__en}.

Presumably there is, in fact, uncountable
number of inner mappings which are not pairwise topologically conjugate.
To explain this,
consider a set $O^\perp_f(x)=\{f^{-n}\left(f^n(x)\right)|\ n\ge 0\}$
for a point $x\in M$
(Definition~\ref{def:perp_trajectory}).
When $f$ is an invertible map, $O^\perp_f(x)=\{x\}$.
However, in the non\discretionary{-}{}{-}invertible case,
$O^\perp_f(x)$ typically is a countable set.
We can consider the main dynamical system  with the map $f$ as the
action of $\zz$.
Then we can consider $O^\perp_f(x)$ as the action of another group,
the iterated monodromy group, that acts locally
in a neighborhood of $O^\perp_f(x)$
using single\discretionary{-}{}{-}valued branches of the family of
multi\discretionary{-}{}{-}value maps $f^{-n}\circ f^n$.
Topological conjugacy has to preserve
not only $\zz$-action on $x$, but $O^\perp_f(x)$ too.
And $O^\perp_f(x)$ as a countable set will have a closure,
which is some fancy closed fractal set
due to aforementioned monodromy action.
This closure should also be preserved by continuity.
A small perturbation of $f$ will cause a countable set $O^\perp_f(x)$
to have another closure, making the perturbed map topologically different.

Consequently, the wandering set of non\discretionary{-}{}{-}invertible inner mappings
exhibits a level of complexity comparable to the nonwandering set of
invertible maps.
The set of equivalence classes of wandering sets up to conjugacy is
most likely uncountable and can not possibly be described.
It is a so called ``wild'' topological problem.
Therefore, we can only compromise and 
reduce complexity by imposing restrictions
on the dynamics in the wandering set.

\subsection{Restrictions on the wandering set.}

As mentioned in Section~\ref{sec:topological_conjugacy_problem_of_wandering_set},
we want to impose restrictions on the wandering set to reduce complexity.
Holomorphic maps are
the most well-studied subclass of non\discretionary{-}{}{-}invertible
inner mappings.
Let us create the restrictions using holomorphic maps as an example.

Let an analytic function $f$ has an attracting critical point
(in the terminology of holomorphic dynamics --- a superattracting point).
One can always choose a local uniformizing parameter $z$ such that
$z=0$ will be a superattracting fixed point of the function $f(z)$.
Then $f(z)$ has the form
\begin{equation}\label{eq:fzn-bottcher}
f(z) = a_n z^n + a_{n+1} z^{n+1} + \dots,
\end{equation}
where $n\ge 2$, $a_n\not=0$.

\begin{theorem}[B\"{o}ttcher's theorem, see~\cite{Milnor1999}]\label{th:Bottcher}
  Let the function $f$ be of the form (\ref{eq:fzn-bottcher}).
  Then there exists a local holomorphic change of
  coordinate $w=\phi(z)$ which conjugates $f$ to the $n$-th power map
  $w\mapsto w^n$ throughout some neighborhood of $\phi(0)=0$.
\end{theorem}

B\"{o}ttcher's theorem has an interesting consequence in
topological dynamics of inner mappings, as pointed
in~\cite{Vla:BRD6:ODESSA2__en}.
Consider a map $z^n$, $n\ge 2$.
It is easy to verify by direct calculation that for any
$z_0\in\cc$ its neutral section $O^\perp(z_0)$ is everywhere dense in the level
curve $|z|=|z_0|$ of the real function $|z|$.
As $O^\perp(x)$ is topological invariant of inner mappings,
then its closure $\overline{O^\perp(x)}$ is also the topological invariant.
For the map $z^n$ it means that the foliation on the level
curves of the function $|z|$ is the topological invariant of $z^n$.
Not only that, the function $|z|$ can be considered as a special
Lyapunov function in the neighborhood of $0$ such that its
foliation on the level curves is invariant under the action of $z^n$:
a circle $|z|=const$ moves onto the circle $|z|=const^n$.

Together with B\"{o}ttcher's theorem~\ref{th:Bottcher} it means that
any superattracting point $p$ of a holomorphic map $f$
has a distinguished Lyapunov function $\phi$ in
a neighborhood of $p$
such that
the foliation on the level curves of $\phi$ is $f$-invariant,
and this foliation itself is a topological invariant of the
superattracting point $p$, which is preserved under topological conjugacy.

Let's write
a purely topological definition of this constraint on the wandering set
which is derived from the behavior of the holomorphic functions
in a neighborhood of an attracting singular point.

\begin{definition}\label{def:Bottcher_point}
  An attracting singular point $p\in M$
  of an inner mapping $f\colon M\to M$
  is called \emph{B\"{o}ttcher attracting point},
  if there exists an open connected neighborhood $V$ of $p$ and a function
  $\phi\colon V\setminus\{p\}\to \rr$ such that
  \begin{enumerate}
  \item\label{enum:Bottcher_point1}
    $|\deg f_{V\setminus\{p\}}|>1$;
  \item
    $\phi$ is monotonically decreases along the partial trajectories of $f$;
  \item
    Each level curve of $\phi$ is homeomorphic to a circle;
  \item
    foliation on the level curves of $\phi$ is $f$-invariant;
  \item\label{enum:Bottcher_point5}
    For any point $x\in V\setminus\{p\}$, $O^\perp(x)$ is dense in the
    level curve $\phi^{-1}(\phi(x))$.
\end{enumerate}
\end{definition}
This definition captures the key properties observed in holomorphic
dynamics near superattracting critical points, but formulated in a
purely topological context applicable to general inner mappings.

As we are only interested in the wandering set,
an attracting singular point can be replaced by an abstract attractor.
In that case
we get a definition of \emph{B\"{o}ttcher component} from~\cite{Vla:BRD4:UMG}.
The paper~\cite{Vla:BRD4:UMG} studied conjugacy of B\"{o}ttcher components,
but not their topological type.

However, these constraints imposed on the wandering set for B\"{o}ttcher
components are too restrictive. In the absence of singular points,
such subsets of the wandering set
belong to single conjugacy class~\cite{Vla:BRD4:UMG}.
Imposing such strict limitations on the wandering
set might limit the generality of the analysis. Ideally, we would like
to consider the questions raised in
Section~\ref{sec:decomposition_of_wandering_set} in the most general
case possible.
Any restriction on the wandering set is a technical step back.
In this paper we take a step forward and
lax the condition~\ref{enum:Bottcher_point5}
in Definition~\ref{def:Bottcher_point},
so that $O^\perp(x)$ do not need to be dense in the level curve.
In that case we also do not need the condition~\ref{enum:Bottcher_point1}
and can allow homeomorphisms as well.

The obtained wandering set restriction we 
call \emph{Pseudo\discretionary{-}{}{-}B\"{o}ttcher components}
by analogy with B\"{o}ttcher components.
Its formal definition will be given in Section~\ref{sec:PseudoBottcher}
as Definition~\ref{def:pseudobottcher_basin1}.
Inner mappings with such restriction
already show an abundance of conjugacy classes
in their wandering set~\cite{Vla:BRD5:ODESSA1__en}.

\subsection{Pseudo-B\"{o}ttcher components.}
\label{sec:PseudoBottcher}

From now on let $f\colon M\to M$ be an inner mapping of
compact two-dimensional manifold $M$ without boundary.

Let $(A, R)$ be an attractor\discretionary{-}{}{-}repeller pair with a basin of attraction $B$
and a region of strict wandering $D$
(Definitions \ref{def:strictly_attracting_set} and
\ref{def:region_of_strict_wandering},
$D=B\setminus \cup_{n\ge 0}f^{-n}(A)=B\setminus A^\perp$).
$D$ is a totally invariant set by construction.
Connectivity components of $D$ are called
components of strict wandering (Definition~\ref{def:component_of_strict_wandering}).

Suppose $S\subset D$ is a totally invariant component of strict wandering.

\begin{definition}\label{def:pseudobottcher_basin1}
  $S$ is called \emph{Pseudo-B\"{o}ttcher} component
  if there exists a
  strictly attracting (Definition~\ref{def:strictly_attracting_set})
  neighborhood $U_0$ of the attractor $A$
  such that
  \begin{enumerate}[leftmargin=*]
  \item
  $A$ is generated by $U_0$, i.e. $A=\cap_{n\ge 0}f^n(\overline{U_0})$;
  \item
  the open set $V_0=U_0\cap S$ is connected;
  \item
    \label{enum:pseudobottcher1_3}
  $V_0$ does not have singular points of $f$;
  \item
    \label{enum:pseudobottcher1_tau0}
    There exists a continuous function $\tau_0\colon V_0 \to (-\infty, 0)$ such that
  \begin{enumerate}[label={\rm\roman*)}]
  \item
    \label{enum:pseudobottcher1_tau0_i}
    $\tau_0$ is a trivial fiber bundle with fiber $S^1$;
    i.e. the foliation on the level curves of $\tau_0$ is regular
    and any level curve
    $\tau_0^{-1}(\lambda)$ is homeomorphic to a circle;
  \item
    \label{enum:pseudobottcher1_tau0_ii}
    $\forall x\in V_0$ $\tau(f^{n}(x))=\tau(x)-n$, $n\in\nn$.
  \item
    \label{enum:pseudobottcher1_tau0_iii}
    $\tau_0$ is constant on neutral sections:\par 
    $\forall x\in V_0$ $\tau_0(O^\perp_{f}(x)\cap V_0)=const$.
  \end{enumerate}
  \end{enumerate}
\end{definition}

\begin{lemma}\label{lm:level_curve_regular_covering}
  In the  Pseudo-B\"{o}ttcher component $S$
  the foliation on level curves of $\tau_0$ is invariant
  under the actions of $f$ and $f^{-1}$ restricted to having images and
  preimages in $V_0$.
  In particular, the following equation holds true:
  $\forall t \in (-\infty,0)$ $\forall n\ge 0$
  $f^n(\tau_0^{-1}(t))=\tau_0^{-1}(t-n)$.
\end{lemma}

\begin{proof}
  Consider the level curves $\gamma_t=\tau_0^{-1}(t)$ and $\gamma_{t-n}=\tau_0^{-1}(t-n)$.
  From the condition~\ref{enum:pseudobottcher1_tau0}-\ref{enum:pseudobottcher1_tau0_ii} of Definition~\ref{def:pseudobottcher_basin1},

  $f^n(\gamma_t) \subseteq \gamma_{t-n}$.
  $\gamma_t$ is closed set. Then $f^n(\gamma_t)$ also is closed set
  because $f^n$ is open closed map.

  If $f^n(\gamma_t)\not=\gamma_{t-n}$, then the image $f^n(\gamma_t)$ has a
  boundary in $\gamma_{t-n}$. Let a point $f^n(x)$ be on the boundary of
  $f^n(\gamma_t)$ in $\gamma_{t-n}$. Then in local coordinates $f^n$
  bends an open segment neighborhood of $x$ in $\gamma_t$ that looks like $(-\delta,\delta)$ to a
  half-open segment in $\gamma_{t-n}$ with $f^n(x)$ as its boundary that looks like $[f^n(x),\varepsilon)$.
  Then $f^n$ is not a local homeomorphism in $x$.
  But from the condition~\ref{enum:pseudobottcher1_3} of
  Definition~\ref{def:pseudobottcher_basin1}, $V_0\cap \Sing(f)=\emptyset$,
  $\Sing(f^n)=f^{-n}(\Sing(f))$ $\Rightarrow$ $V_0\cap \Sing(f^n)=\emptyset$.
  It means that $f^n$ is a local homeomorphism in $V_0$.
  We got a contradiction.

  Therefore, $f^n(\gamma_t)=\gamma_{t-n}$. As
  $\gamma_t=\tau_0^{-1}(t)$ and $\gamma_{t-n}=\tau_0^{-1}(t-n)$,
  we got the equation $f^n(\tau_0^{-1}(t))=\tau_0^{-1}(t-n)$.
\end{proof}

\begin{remark}\label{rem:pseudobottcher_basin1_v_condition}
  Since $M$ is compact, the condition~\ref{enum:pseudobottcher1_3} of
  Definition~\ref{def:pseudobottcher_basin1} is redundant because the number of singular points of $f$
is at most finite and we can always shrink $V_0$ to avoid them.
\end{remark}
\begin{remark}
From the condition~\ref{enum:pseudobottcher1_tau0}-\ref{enum:pseudobottcher1_tau0_ii} of Definition~\ref{def:pseudobottcher_basin1},
$\forall x\in V_0$
$\tau_0(f(x)) < \tau_0(x)$, which means that $\tau_0$
is a Lyapunov function for the restriction $f|_{V_0}$.
\end{remark}

\subsection{Expansion of $\tau_0$ onto the Pseudo-B\"{o}ttcher component.}

Let $M$ be a compact connected 2-dimensional manifold without boundary.
Let $f\colon M\to M$ be an inner mapping of $M$.
Let $S$ be a totally invariant Pseudo-B\"{o}ttcher component
of strict wandering of $f$ (Definition~\ref{def:pseudobottcher_basin1}).
\begin{lemma}\label{lm:pseudo_Butcher_timeline_coordinate}
  There exists a function $\tau\colon S\to (-\infty,+\infty)$ such that
  \begin{enumerate}[label={\rm \roman*)}, align=right, leftmargin=*]
  \item
    \label{emum:lm_timeline_i}
    $\forall x\in V_0$ $\tau(x) = \tau_0(x)$.
  \item
    \label{emum:lm_timeline_ii}
    $\tau(f^{n}(x))=\tau(x)-n$, $n\in\zz$, i.e.
    partitioning on level curves of $\tau$ is invariant under the
    actions of $f$ and $f^{-1}$.
  \item
    \label{emum:lm_timeline_iii}
    For any $c\in (-\infty,+\infty)$ the level curve
    $\tau^{-1}(c)$ is neutrally saturated.
    As a result, $\tau$ is constant on neutral sections.
  \end{enumerate}
\end{lemma}

\begin{proof}
  There is already a function $\tau_0(x)$ on $V_0\subset S$
  from Definition~\ref{def:pseudobottcher_basin1}.
  Let's expand $\tau$ from $V_0$ onto $S$ in 4 steps.

  1) To satisfy the condition~\ref{emum:lm_timeline_i},
  put $\tau(x)=\tau_0(x)$ for all $x\in V_0$.

  2) Consider the set $V_0^\perp=\bigcup_{n\geq 0} f^{-n}(f^n(V_0))$.
  By Definition~\ref{def:pseudobottcher_basin1} $V_0$ does not
  contain singular points of $f$.
  Therefore, for all $n\geq 0$ the map $f^n$
  restricted to $V_0$ is a regular covering, $V_0$ is a distinct
  sheet of the $n$-sheeted covering $f^n$ restricted to the set $f^{-n}(f^n(V_0))$
  and the set $V_0$ is distinguished as a connected component in the set $f^{-n}(f^n(V_0))$.
  As a result, $V_0$ is connected component of $V_0^\perp$.
  Note that $V_0$ and $V_0^\perp$ either differ or coincide.

  Put for all $x\in V_0^\perp$ $\tau(x)=\tau_0(O^\perp(x)\cap V_0)$.
  This formula is well defined because $\tau_0$
  is constant on neutral sections in $V_0$
  (Definition~\ref{def:pseudobottcher_basin1} ~\ref{enum:pseudobottcher1_tau0}-\ref{enum:pseudobottcher1_tau0_iii}).
  Note that in this step the condition~\ref{enum:pseudobottcher1_tau0}-\ref{enum:pseudobottcher1_tau0_ii}
  of Defintiton~\ref{def:pseudobottcher_basin1}
  turns into the condition~\ref{emum:lm_timeline_ii} of this Lemma.

  3) By continuity, put $\tau(x)=0$ for all $x\in \overline{V_0^\perp}\setminus V_0^\perp$.
  On this step the condition ii) $\tau(f^{n}(x))=\tau(x)-n$ still hold.

  4) Define $\tau$ onto $S$:
  if $f^{-m}\left(f^n\left(x\right)\right)\in \overline{V_0^\perp}$,
  $m,n\in \zz^+$ then
  \begin{equation}\label{eq:tau_S}
    \tau\left(f^{-m}\left(f^n\left(x\right)\right)\right)=\tau_0(x)+m-n.
  \end{equation}
  $S$ does not have periodic points of $f$
  (if the basin of attraction $B$ contains a periodic
  point $p$ of $f$ then $p$ belongs to the attractor $A$ by construction).
  Therefore, $\forall x\in S$ $O(x)/O^\perp(x)\approx \zz$.
  From the formula (\ref{eq:tau_S}), it follows that
  $\forall x\in S$ the function $\tau$ has the same value on $O^\perp(x)$,
  i.e. the function $\tau$ is constant in $S$ on neutral sections. 
  We can consider the function $\tau$ as a function defined on
  equivalence classes of the partitioning of $S$ by $O^\perp$.
  By definition, the function $\tau$ is coordinated with the action of
  $f$ on $S$. Using commutative diagrams, we can write it as\\
  \xymatrix{
    O(x) \ar[d]^{[O/O^\perp]} \ar[r]^{f} & O(x) \ar[d]^{[O/O^\perp]} &
    S/O^\perp \ar[d]^\tau \ar[r]^{f} & S/O^\perp \ar[d]^\tau &
    S \ar[d]^\tau \ar[r]^{f} & S \ar[d]^\tau
    \\
    \zz \ar[r]^{-1} & \zz &
    \rr \ar[r]^{-1} & \rr &
    \rr \ar[r]^{-1} & \rr
  }

  It follows that for any point $x$ for any
  $m_1,n_1,m_2,n_2\in \zz^+$ such that
  both $f^{-m_1}\left(f^{n_1}\left(x\right)\right)\in \overline{V_0^\perp}$,
  and
  $f^{-m_2}\left(f^{n_2}\left(x\right)\right)\in \overline{V_0^\perp}$,
  $\tau$ has the same value in $x$.
  Thus, $\tau$ is well defined in $x$.
  The function $\tau$ is defined in (\ref{eq:tau_S}) as the composition
  of continuous functions, thus for any $x\in S$ $\tau$ is continuous
  in $x$.
  The condition~\ref{emum:lm_timeline_ii} follows from the formula (\ref{eq:tau_S}).
  The function $\tau$ is well defined in $S$ with the
  properties~\ref{emum:lm_timeline_i}-~\ref{emum:lm_timeline_iii}.
  This completes the proof.
\end{proof}

\begin{corollary}\label{cor:level_curve_regular_covering_S}
  $f^n(\tau^{-1}(t))=\tau^{-1}(t-n)$.
\end{corollary}

\section{Base annulus and partitioning}
Let $f\colon M\to M$ be an inner mapping of compact two-dimensional manifold $M$
and $S\subset M$ be a totally invariant
Pseudo-B\"{o}ttcher component of strict wandering.

This section lays the groundwork for understanding the topological
structure of $S$ by establishing a partitioning of $S$ into
multiple connected compact surfaces with boundary, called atoms.

This partitioning is analogous to the
Birkhoff's concept of ``fundamental neighborhoods'' \cite{BirkgoffDS},
used in studying homeomorphisms.
However, ``fundamental neighborhoods'' do not generalize easily to
arbitrary non\discretionary{-}{}{-}invertible maps. In our
construction, we rely heavily on the properties of level curves of the function
$\tau_0$ from Definition~\ref{def:pseudobottcher_basin1}.

We start from studying topology of single atoms. Then we focus on atom
groups, general ``molecules'' and special molecules called ``chains'',
which are used to study the action of $f$ on $S$.
Later, we consider how this partitioning is connected with the topology of $S$ as a whole.

These results will be instrumental in our analysis of $S$'s topology in Section~\ref{sec:pseudo_bottcher_topology}.

\subsection{Partitioning of $V_0$.}\label{sec:PartitioningV0}

\begin{figure}[htpb]
    \includegraphics[scale=0.4]{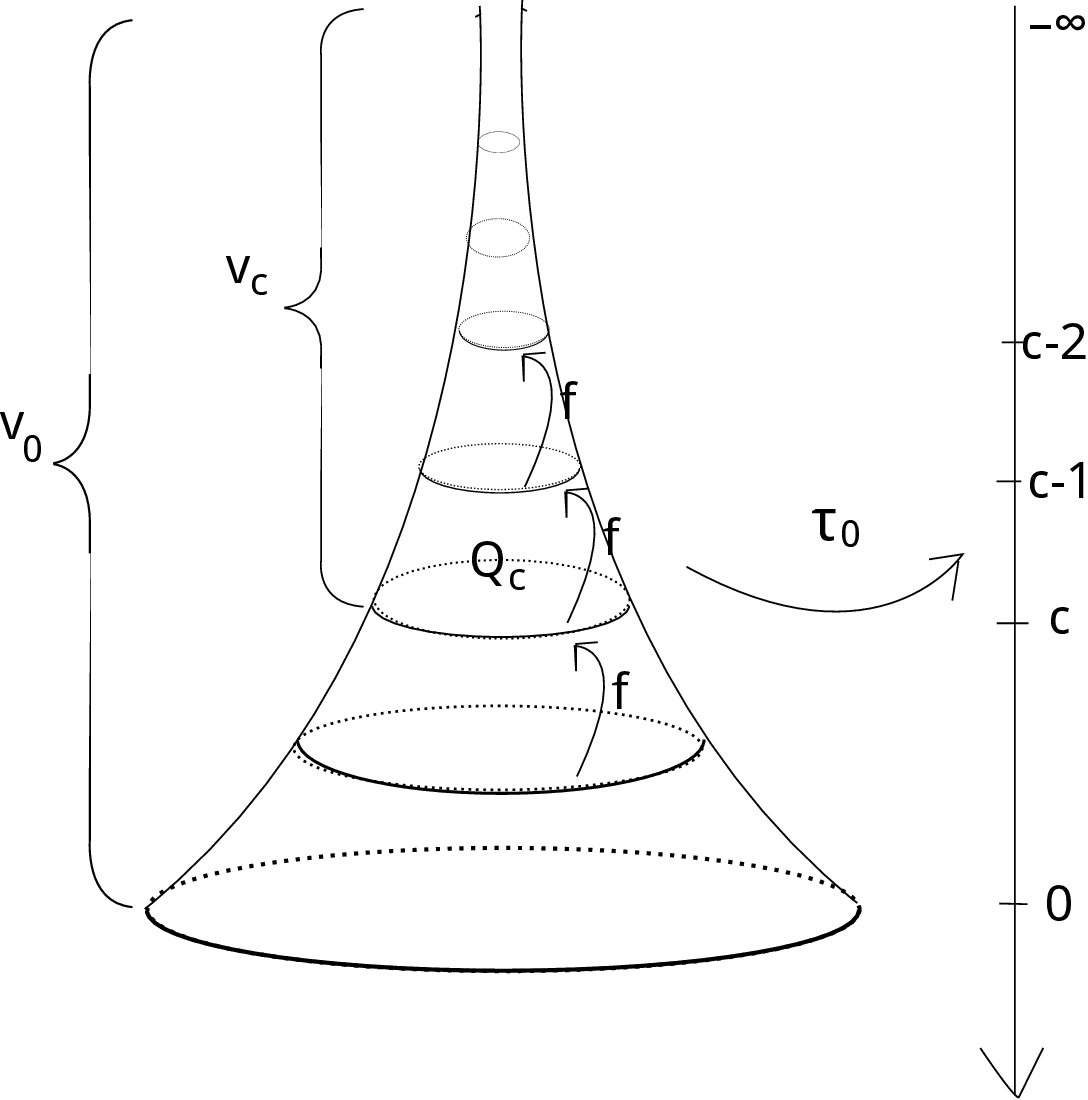}
\caption{
$V_0$, $V_c$ and $Q_c$.
}
\end{figure}

Consider the neighborhood $V_0$ and the function $\tau_0\colon V_0\to \rr$ from Definition~\ref{def:pseudobottcher_basin1}.

Define
$V_c:=\tau_0^{-1}\left(\left(-\infty,c \right)\right)$ 
and $Q_c:=\tau_0^{-1}\left(\left[c-1,c\right]\right)$ for any $c\in (-\infty,0)$.
$Q_c\subset \overline{V_c}$, $V_c\subset V_0$.
By Definition~\ref{def:pseudobottcher_basin1}, $V_c$ and $Q_c$ are
homeomorphic to the open and closed annulus respectively.

Choose the orientation of $V_c$ towards the attractor $A$.
As a closed topological annulus, $Q_c$ is closed connected set.
Its boundary $\partial Q_c$ has 2 connected components:
$\partial^+ Q_c=\tau_0^{-1}(c)$ and $\partial^- Q_c=\tau_0^{-1}(c-1)$.
Each $\partial^+ Q_c$ and $\partial^- Q_c$ is homeomorphic to a circle.
Let's call $\partial^+ Q_c$ and $\partial^- Q_c$ \emph{external} and \emph{internal} part of $\partial Q_c$.

$V_c$ is naturally partitioned with the images of $Q_{c}$ (Lemma~\ref{lm:level_curve_regular_covering}):
$$V_c=\cup_{n\in\zz^+}Q_{c-n}=\cup_{n\in\zz^+}f^n(Q_{c}).$$
Note that
not only $V_0$ with function $\tau_0$, but for any $c\in (-\infty,0)$
$V_c$ with function $\tau_c=\tau_0+c$ also
satisfy the conditions of Definition~\ref{def:pseudobottcher_basin1}.
In particular, by construction, $V_c$ is restriction of a strictly attracting
neighborhood of the attractor $A$ on $S$.
Then $\forall x \in S$ $\exists m,n\in Z^+$:
$f^n(x)\in Q_{c-m}\subset V_c$, $Q_{c-m}=f^m(Q_{c})$ $\Rightarrow$
$\forall x\in S$ $O(x)\cap Q_{c}\not=\emptyset$.
Therefore,
\begin{equation}\label{f:S_fundII_lmpart}
  S
  = \cup_{l,m\ge 0} f^{-l}\left(Q_{c-m}\right)%
  = \cup_{l,m\ge 0} f^{-l}\left(f^m\left(Q_c\right)\right)%
.
\end{equation}

Conceptually, partitioning of $V_0$ is very easy to understand.
We partition $\rr$ into segments
$[c-1-n,c-n]_{n\in\zz}$ of length 1. Their preimages with $\tau_0^{-1}$ become
partitioning of $V_0$ into annuli.
The only complexity there is using a constant $c$. $c$ is required to
offset the boundaries of the annuli from the trajectories of singular points.

\begin{lemma}
For each point $p\in \Sing(f)$, such that $O(p)\cap V_o\not=\emptyset$,
the function $\tau_0\colon V_0\to \rr$ maps the set
$O(p)\cap V_o$ to a discrete set in $\rr$.
\end{lemma}

\begin{proof}
If $p\in \Sing(f)$, then $p\not\in V_0$.
(Defintiton~\ref{def:pseudobottcher_basin1}-\ref{enum:pseudobottcher1_3}).
If $O(p)\cap V_o\not=\emptyset$, then there exists the smallest $m>0$
such that $f^m(p)\in V_o$.
From the conditions~\ref{enum:pseudobottcher1_tau0}-\ref{enum:pseudobottcher1_tau0_ii}
and~\ref{enum:pseudobottcher1_tau0}-\ref{enum:pseudobottcher1_tau0_iii}
of Defintiton~\ref{def:pseudobottcher_basin1},
$\tau_0(O(p)\cap V_o)=\{\tau_0(f^m(p))-n | n\in \nn\}$.
\end{proof}

Let $O\left(\Sing\left(f\right)\right):=\cup_{p\in \Sing(f)} O(p)$
be the full trajectory of the set of singular points.
The set $\Sing(f)$ is finite.
Then
\begin{corollary}
  The function $\tau_0\colon V_0\to \rr$ maps the set
  $O\left(\Sing\left(f\right)\right)\cap V_o$ to a discrete set in $\rr$.
\end{corollary}

\subsection{Introducing atoms.}\label{sec:atoms}

Let's select a constant $c$ for the set $Q_c$
such that $c\not \in \tau_0\left(O\left(\Sing\left(f\right)\right)\right)$.
In that case, the boundary $\partial Q_c$ has no intersection with
trajectories of singular points.

Let's call the selected $Q_c$ a \emph{base annulus} and
build a partitioning of $S$ based on $Q_c$.

\begin{definition}\label{def:atom_comp_fundIIloc}
  Connected components of the sets
  $f^{-l}(f^m(Q_c))$, where $l,m\ge 0$, are called \emph{atoms}.
\end{definition}

Note that $f^{-l}(f^m(Q_c))\subset f^{-(l+1)}(f^{m+1}(Q_c))$.
Thus, for any given atom $A$
the numbers $l,m\ge 0$ are defined up to the difference $m-l$.

\begin{definition}\label{def:atom_comp_generation}
  For an atom $A\subset f^{-l}(f^m(Q_c))$ the number $m-l$ is called
  $A$'s \emph{generation}.
\end{definition}

Atoms defined by numbers $l,m$
do not change when defined by numbers $l+k,m+k$.
One way to prove it is to consider that the multivalue maps
$f^{-n}\circ f^{n}$ are the identity maps locally, as
they might introduce new preimages under
$f^{-n}$, but they do not modify the existing preimages
when $n$ increases.
Another way is to define atoms is to use the function $\tau$ from
Lemma~\ref{lm:pseudo_Butcher_timeline_coordinate}:

\begin{definition}\label{def:atom_comp_tau}
  Connected components of the sets
  $\tau^{-1}([c+n-1,c+n])$, $n\in \zz$ are called \emph{atoms}
  (of \emph{generation} $n$).
\end{definition}

From the condition~\ref{emum:lm_timeline_ii} of
Lemma~\ref{lm:pseudo_Butcher_timeline_coordinate},
Definitions~\ref{def:atom_comp_fundIIloc} and \ref{def:atom_comp_tau}
are equivalent: any atom $A$ which is a connected component of
$f^{-l}(f^m(Q_c))$ is a connected component of
$\tau^{-1}([c-n-1,c-n])$, where $n=m-l$ is the generation of the atom $A$.

There exist a lot of different partitionings of $S$ that depend on
\begin{enumerate}
  \item
the choice of the neighborhood $V$ and the function $\tau_0$ in the
definition~\ref{def:pseudobottcher_basin1} of Pseudo-B\"{o}ttcher component.
  \item
the choice of $c$ in $Q_c$, that defines the base annulus.
\end{enumerate}
Moreover, those partitionings might be different topologically:
for example, if $S$ has 2 singular points sufficiently close,
then there could exist one partitioning where those 2 singular points
belong to a single atom, and another partitioning,
where those 2 singular points belong to two different atoms.
However, partitionings are just a tool to study $S$.
In the end, the results will not depend on the choice of partitioning.

From now on, we fix the previously selected constant $c$, the base annulus $Q_c$
and the corresponding partitioning of $S$ so that any atom we speak
of can only belong to the partitioning based on $Q_c$. 

\begin{definition}\label{def:ext_int_boundary_atom}
  For an atom $A$ define the \emph%
  {internal} and \emph%
  {external} parts of the boundary $\partial A$ of atom $A$ respectively
  as $\partial^- A=\tau^{-1}(c-n-1)\cap A$ and $\partial^+ A=\tau^{-1}(c-n)\cap A$.
\end{definition}

\begin{definition}\label{def:adjacent_atom}
  Two atoms $A_1$ and $A_2$
  are called \emph%
  {adjacent}, if they have
  a non-empty intersection of their boundaries.
\end{definition}
By construction, a common part of the
boundary of adjacent atoms is a part of external boundary for one atom and
a part of internal boundary for another.
This intersection of boundaries induces a natural equivalence relation:
A point $x$ on the external part of the boundary of one atom is
equivalent to a point $y$ on the internal part of the boundary of the
other atom if $x$ and $y$ are mapped into the same point by natural embedding of
their atoms into $S$.

\begin{definition}\label{def:atom_gluing}
Let's call this natural equivalence a \emph{gluing} of atoms.
\end{definition}

\subsection{Atom's topology and the Riemann-Hurvitz formula.}
Recall the Riemann-Hurvitz formula (see.~\cite{Hovanski2007en}).
Let $M_1$, $M_2$ be surfaces and $h\colon M_1\to M_2$ be a branched covering.
Then
\begin{equation}\label{eq:RiemannHurvitz}
  \chi (M_1)=|\deg(h)|\cdot \chi (M_2)-\sum _{{p\in \Sing(h)}}(e_{p}-1)
\end{equation}
where $\chi (M)$ is the Euler characteristic of a surface $M$.

Consider an atom $A$. By atom's definition,
$A=f^{-l}(f^m(Q_c))$ for a $l,m\in Z^+$.
Then $f^l(A)=f^m(Q_c)=Q_{c-m}$.
Therefore,
$A$ is a branched covering of the closed annulus $Q_{c-m}$ by $f^l$.

$\Sing(f^l)=f^{-l}\left(\Sing\left(f\right)\right)\subset O\left(\Sing\left(f\right)\right)$.
By the choice of $c$,
$c\not \in \tau_0\left(O\left(\Sing\left(f\right)\right)\right)$.
Then $\partial A\cap \Sing(f^l)=\emptyset$ and
both $\partial^+ A$ and $\partial^- A$
have no intersection with the trajectories of singular points of $f$ and $f^l$.
Therefore, they are
regular coverings (by $f^l$)
of $\partial^+ Q_{c-m}$ and $\partial^- Q_{c-m}$ respectively.

Consider the map $f^l|_A\colon A\to Q_{c-m}$.
For the annulus $Q_{c-m}$ the Euler characteristic
$\chi (Q_{c-m})$ is 0.
Using the Riemann-Hurvitz formula~\ref{eq:RiemannHurvitz},
we obtain
the following formula and Lemma~\ref{lm:atom_planar}
as the consequence:

\begin{equation}\label{eq:RiemannHurvitz_atom}
  \chi (A)=-\sum _{{p\in A\cap \Sing(f^l)}}(e_{p}-1)
\end{equation}

\begin{lemma}\label{lm:atom_planar}
If $A\cap \Sing(f^l)=\emptyset$ then $A$ is regular covering of
$Q_{c-m}$, $\chi (A)=0$,
and $A$ itself is also homeomorphic to a closed annulus.

Otherwise, $\chi (A) <0$,
$A$ is a planar set (a disc with holes), a
two-dimensional manifold with boundary. Each connected component
of $\partial A$ is homeomorphic to a circle, and $\partial A$ has at
least 3 or more connected components (the exact number depends on
the set $\Sing(f^l)\cap A$).
\end{lemma}

Consider the function $\tau$ from
Lemma~\ref{lm:pseudo_Butcher_timeline_coordinate}, restricted on $A$.

\begin{lemma}\label{lm:atom_trivial_fiber_bundle}
  If $A\cap \Sing(f^l)=\emptyset$, then $\tau|_A: A\to ([c+l-m-1,c+l-m])$
  is a trivial fiber bundle
  with fiber $S^1$.
\end{lemma}

\begin{proof}
  We will now prove that if $A \cap \Sing(f^l) = \emptyset$, then the
  restriction of $\tau$ to $A$ defines a trivial fiber bundle.
  According to Lemma~\ref{lm:atom_planar},
  $A$ is a regular covering of $Q_{c-m}$.
  Similar to the proof of
  Lemma~\ref{lm:level_curve_regular_covering},
  we get $\forall t\in [c+l-m-1,c+l-m]$ $f^l(\tau^{-1}(t)\cap A)=\tau_0^{-1}(t-l+m)$.
  Then $\tau^{-1}(t)\cap A \subset f^{-l}(\tau_0^{-1}(t-l+m))$.
  $\tau_0^{-1}(t-l+m)\approx S^1$ $\Rightarrow$ $f^{-l}(\tau_0^{-1}(t-l+m))$
  is a finite number of $S^1$.

  Lamination $\mathcal{L}$ on level curves of $\tau$ on $A$ is a regular
  covering of the regular foliation on level curves of $\tau_o$ on
  $Q_{c-m}$. Therefore, $\mathcal{L}$ itself is regular foliation where
  each level curve is homeomorphic to a finite number of $S^1$.
  Its boundaries $\partial^+ A$ and $\partial^- A$ consist of one
  $S^1$ each.
  By continuity, each level curve in $\mathcal{L}$ is also homeomorphic to one $S^1$
  which gives us the statement of the Lemma.
\end{proof}

By the choice of $c$,
the boundary of any atom belongs to a regular level curve of the function $\tau$.
Therefore, when restricted on the boundary, $f$ is a regular covering
and connected components of the boundary are circles.
The number of connected components is finite because $f$ is a
finite-to-one map.

\begin{definition}\label{def:atom_boundary_type}
  An atom $A$
  is called to be of
  \emph%
  {(boundary) type $(a,b)$}
  if $\partial^- A$ consists of $a$ circles
  and $\partial^+ A$ consists of $b$ circles.
\end{definition}

This lemma establishes how the boundary type of an atom changes under
$f^n$ depending on the presence or absence of singular points within the atom.
\begin{lemma}\label{lm:atom_ab_type_map_classification}
  Let $f^n$ takes an atom $A_1$ of type $(a_1, b_1)$ to
  an atom $A_2$ of type $(a_2, b_2)$.

  If the atom $A_1$ does not have singular points of $f^n$,
  then

  1-1) if the atom $A_2$ is of type $(1, 1)$,
  then $A_1$ is of type $(1, 1)$ too;

  1-2) if the atom $A_2$ is of type $(a_2, b_2)$, where $a_2\not=1$ or $b_2\not=1$,
  then $A_1$ is of type $(a_2, b_2)$ too and $f^n$ is a homeomorphism
  from $A_1$ to $A_2$.

  If the atom $A_1$ does have singular points of $f^n$,
  then 2) $a_1> a_2$, or $b_1> b_2$ or both.
\end{lemma}
The proof is by direct application of the Riemann-Hurvitz formula~\ref{eq:RiemannHurvitz}
and the Lemma~\ref{lm:atom_planar}.

\begin{corollary}\label{cor:atom_ab_type_map_foliation}
  Let $f^n$ takes an atom $A_1$ 
  to an atom $A_2$. 
  If the atom $A_1$ does not have singular points of $f^n$,
  then $A_1$ and $A_2$ have homeomorphic partitionings on level curves of $\tau$.
\end{corollary}

\subsection{Molecules.}

\begin{definition}\label{def:molecule_fundIIloc}
  A connected subset of $S$ which is obtained by gluing different
  atoms together is called \emph{molecule}.
\end{definition}

A molecule that consists of finite number of atoms is called a finite molecule.

\begin{lemma}\label{lm:finite_molecule_planar}
  A finite molecule is a planar set.
\end{lemma}

\begin{proof}
  Assume the converse. Let a molecule $D$ be a surface with non-zero
  genus. It does not matter whether the genus is oriented or not.
  Since the molecule $D$ is finite, its image $\tau(D)$ is bounded:
  there exist $m,n\in Z$ such that $D\subset \tau^{-1}([m,n])$.
  According to the Riemann-Hurvitz formula~\ref{eq:RiemannHurvitz},
  a surface with genus can only cover a surface with genus.
  Therefore, for all $k\ge 0$ the molecule $f^{k(n-m)}(D)$ also is a surface with genus.
  Then $S$ is a surface of infinite genus
  because $\forall k\ge 0$ $f^{k(n-m)}(D)\subset S$.
  But $S$ is a part of compact oriented surface $M$.
  This contradiction proves the Lemma.
\end{proof}

If an infinite molecule $E$ has a genus, then $E$ contains a finite
submolecule $D$ that has a genus, because both a handle and a
M\"{o}bius strip
are compact sets. As a consequence, we got
\begin{corollary}\label{cor:any_molecule_planar}
  An infinite molecule is a planar set.
\end{corollary}

\begin{corollary}\label{cor:atom2disconnect}
  Any atom disconnects $S$ into multiple connected components.
\end{corollary}

\begin{lemma}\label{lm:one_adjacent_boundary}
  An atom $A$ and a molecule $B$ can only be adjacent with a single common
  connected component of their boundary.
\end{lemma}
\begin{proof}
  Assume the converse.
  Suppose we consider the gluing of
  $A$ and $B$
  using 2 or more components of their boundary.
  But this will create a handle
  and a resulting molecule would get a genus,
  contradicting Lemma~\ref{lm:finite_molecule_planar}.
\end{proof}

\begin{lemma}\label{lm:molecule_intersection}
  Let $B_1$, $B_2$ be molecules.
  Then the set $B_1\cap B_2$, if not empty, is either molecule
  or a single common connected component of their boundary.
\end{lemma}

The proof's idea is the same as in Lemma~\ref{lm:one_adjacent_boundary}.

\subsection{A chain of atoms.}

\begin{definition}\label{def:forward_chain_of_atoms}
    A sequence of atoms $(A_i)_{i\in I}$, where $I$ is an interval in $\zz$,
    is called a \emph{forward chain} if
    $\partial^- A_i \cap \partial^+ A_{i+1} \neq \emptyset$ for all $i \in I$.
\end{definition}


By Definition~\ref{def:atom_comp_generation}, adjacent atoms in the forward chain differ in
generation by 1:
if an atom $A_{i}$ belongs to generation $k_i$, then the adjacent atom $A_{i+1}$ belongs to generation $k_i+1$.

\begin{definition}\label{def:backward_chain_of_atoms}
  A sequence of atoms $(A_i)_{i\in I}$, where $I$ is an interval in $\zz$,
  is called a \emph{backward chain} if
  $\partial^- A_{i}\cap \partial^+ A_{i-1}\not=\emptyset$ for all $i\in I$.
\end{definition}

By definition, a backward chain is a forward chain where atoms are
enumerated in reverse order.
We will assume, unless explicitly stated otherwise, that
in a forward chain atoms are enumerated in the
direction of decreasing $\tau$ values (towards attractor),
in a backward chain -- towards repeller.

Chains can be classified as follows:
\begin{itemize}[leftmargin=*]
\item \textbf{Finite chain:}
  $I$ is finite interval of integers, e.g., $0\dots n$.
\item \textbf{Semi\discretionary{-}{}{-}infinite chain:}
  $I$ is semi\discretionary{-}{}{-}infinite interval of integers, e.g., $-\infty \dots 0$.
\item \textbf{Infinite chain:} $I=\zz$.
\end{itemize}
Consecutive atoms in a chain form a connected set called a
molecule. To emphasize the connection between a chain $(A_i)_{i\in I}$
and its corresponding molecule $C$, we write $C=\cup_{i\in I} A_i$.

\begin{lemma}\label{lm:chain_atom_to_atom}
  Let $(A_0, A_1)$ be a forward chain of atoms.
  If $f(A_{0})\cap A_{1}\neq\emptyset$, then $f(A_{0})=A_{1}$.
\end{lemma}

\begin{proof}
  Suppose $A_0$ belongs to generation $k_0$ (Definition~\ref{def:atom_comp_generation}).
  By Definition~\ref{def:atom_comp_tau}, $A_0$ is a connected component of $\tau^{-1}([c-k_0-1,c-k_0])$.
  Consequently, $A_1$ belongs to generation $k_0+1$ and is a connected component of $\tau^{-1}([c-k_0-2,c-k_0-1])$.

  From Corollary~\ref{cor:level_curve_regular_covering_S},
  we know that $f(\tau^{-1}([c-k_0-2,c-k_0-1]))=\tau^{-1}([c-k_0-1,c-k_0])$.
  Due to the continuity of $f$, $f(A_0)$ maps to a single connected
  component of $\tau^{-1}([c-k_0-2,c-k_0-1])$.
  Since $f(A_{0})\cap A_{1}\neq\emptyset$, it follows that $f(A_{0})=A_{1}$.
\end{proof}

\begin{lemma}\label{lm:train_of_atomsN}
  Let $(A_i)_{i=0}^n$ be a forward chain of atoms.
  If $f(A_{n-1})=A_n$, then $f(A_{i-1})=A_i$ for all $i=1, \dots, n$.
\end{lemma}

\begin{proof}
  Let $A_0$ belong to generation $k_0$.
  Then, by definition, $A_i$ belongs to generation $k_0 + i$ for all $i=0, \dots, n$.
  According to Definition~\ref{def:atom_comp_tau}, atoms of generation
  $k$ are connected components of $\tau^{-1}([c-k-1,c-k])$.
  By Lemma~\ref{lm:level_curve_regular_covering}, $f$ maps these atoms
  onto atoms of generation $k+1$ (which are connected components of
  $\tau^{-1}([c-k-2,c-k-1])$).

  We proceed by induction. The base case, $f(A_{n-1})=A_n$, is given.
  Assume that $f(A_{m-1})=A_m$ for some $m \in \{1, \dots, n\}$.

  To show that $f(A_{m-2})=A_{m-1}$, consider the set
  $\Delta = \partial^+ A_{m-1} \cap \partial^- A_{m-2}$.
  Since $\Delta \subset A_{m-1}$ and $f(A_{m-1})=A_m$, it follows that
  $f(\Delta) \subset A_m$. Moreover, as $\Delta \subset A_{m-2}$, we
  have $f(A_{m-2}) \cap A_{m-1} \neq \emptyset$.
  By Lemma~\ref{lm:chain_atom_to_atom}, this implies $f(A_{m-2})=A_{m-1}$.

  Therefore, by induction, $f(A_{i-1})=A_i$ for all $i=1, \dots, n$.
\end{proof}

\subsection{Main and auxiliary atoms.}

$S$ is connected totally invariant set.

A totally invariant set of non\discretionary{-}{}{-}invertible maps
can be represented as a union of an invariant subset and a
sequence of possibly empty consecutive preimages of this invariant subset.

Generally,
this representation is not unique and depends on the choice of the invariant subset.
However, within the fixed partitioning of $S$, we can unambiguously classify atoms
as either part of the core invariant subset, called main atoms, or part of preimages of
the core invariant subset, called auxiliary atoms.

\begin{definition}\label{def:main_auxiliary_atoms}
  An atom $A$ is called \emph{main} if the pair $(A, f(A))$ forms a forward chain of atoms.
  Otherwise, it is called \emph{auxiliary}.
\end{definition}

Atoms contained in $\overline{V_c}$ provide an example of main atoms.

\begin{definition}
  A chain consisting entirely of main atoms is called a \emph{main chain}.
  A chain consisting entirely of auxiliary atoms is called an \emph{auxiliary chain}.
  Otherwise, the chain is called a \emph{mixed chain}.
\end{definition}

\begin{lemma}\label{lm:train_of_atoms0}
  Let $A$ be a main atom.
  Then $f(A)$ is also a main atom.
\end{lemma}

\begin{proof}
  Since $A$ is a main atom, by definition, $(A, f(A))$ is a forward chain.
  This implies that $\partial^- A \cap \partial^+ f(A) \neq \emptyset$.
  Applying $f$ to both sides of this intersection, we get:
  \[
  f(\partial^- A\cap \partial^+ f(A)) = f(\partial^- A)\cap
  f(\partial^+ f(A)) = \partial^- f(A)\cap \partial^+ f(f(A))
  \neq\emptyset.
  \]
  Therefore, $(f(A), f(f(A)))$ is also a forward chain, making $f(A)$ a main atom.
\end{proof}

\begin{corollary}\label{cor:f_of_main_is_main}
  Let $C = \bigcup_{i=0}^n A_i$ be a main forward chain of atoms.
  Then $f(C)$ is also a main forward chain of atoms.
\end{corollary}

\begin{lemma}\label{lm:main_and_auxiliary_atoms}
  Let $A$ be a main atom. Then:
  \begin{enumerate}[leftmargin=*]
  \item\label{enum:main_and_auxiliary_atoms_1}
    Any atom $A_i$ adjacent to $A$ via $\partial^+ A$ is main.
  \item\label{enum:main_and_auxiliary_atoms_2}
    For atoms adjacent to $A$ via $\partial^- A$, $f(A)$ is main, and the rest are auxiliary.
  \end{enumerate}
\end{lemma}

\begin{proof}
  (\ref{enum:main_and_auxiliary_atoms_1}).
  This follows directly from Lemma~\ref{lm:train_of_atomsN}.
  Since $\{A$, $f(A)$, $f(f(A))\}$ is a main forward chain by Lemma~\ref{lm:train_of_atoms0},
  any atom adjacent to $A$ via $\partial^+ A$ must also be main.

  (\ref{enum:main_and_auxiliary_atoms_2}).
  Let $B$ be an atom such that $B \neq f(A)$ and $\partial^- A \cap \partial^+ B \neq \emptyset$.
  By Corollary~\ref{cor:atom2disconnect}, $A$ divides $S$ into connected components.
  Denote the connected component of $S \setminus A$ containing
  $f(f(A))$ and $B$ as $K_{f(f(A))}$ and $K_B$, respectively.
  Since $A$ maps onto $f(A)$, $B$ maps onto $f(f(A))$, and consequently,
  $K_B$ maps onto $K_{f(f(A))}$. Therefore, $B$ is auxiliary.
\end{proof}

\begin{corollary}\label{cor:backward_chain_is_main}
  Let $(A_0, A_1, \dots)$ be a backward chain such that $A_0$ is main.
  Then every atom in the chain is main.
\end{corollary}

\begin{corollary}\label{cor:forward_chain_is_auxiliary}
  Let $(A_0, A_1, \dots)$ be a forward chain such that $A_0$ is auxiliary.
  Then every atom in the chain is auxiliary.
\end{corollary}

\begin{corollary}\label{cor:mixed_chain_split}
  In any forward mixed chain there exists a main atom $A$ such that
  any  atom preceding $A$ in the chain is main and any atom following $A$ in the chain is auxiliary.
\end{corollary}

\begin{corollary}\label{cor:any_atom_of_S_is_main}
  If the Pseudo-B\"{o}ttcher component $S$ contains no atoms of type $(k, *)$ where $k > 1$,
  then every atom in $S$ is main.
\end{corollary}

\begin{lemma}\label{lm:Y_seq_is_infinite}
  Let $\{A_n\}_{n \geq 0}$ be a main backward infinite chain such that
  the first atom $A_{0}$ is of type $(1,*)$, but
  $A_{i}$ are of type $(a_i,1)$ with $a_i \ge 2$ for all $i > 1$.
  Then $f$ has infinitely many singular points in $S$.
\end{lemma}

\begin{proof}
  As $A_0$ is of type $(1, *)$,
  the number of connected components in the external boundary $\partial^+$ increases from 1
  in $\partial^+ A_0$ to $a_1 \geq 2$ in $A_1$.
  Therefore,
  the map $f\colon A_1 \to A_0$ is at least a 2-sheeted covering.
  By the Riemann-Hurwitz formula (Equation~\ref{eq:RiemannHurvitz}),
  $A_1$ must contain a singular point.

  We proceed by induction. Assume that $f\colon A_{n-1} \to A_{n-2}$
  is at least an $n$-sheeted covering, implying that $f|_{\partial^+ A_{n-1}}$
  is also at least an $n$-sheeted covering.
  Since $A_{n-1}$ is of type $(*, 1)$, $\partial^+ A_{n-1}$ is connected.

  Consider $A_{n}$.
  Atoms $A_{n-1}$ and $A_{n}$ are adjacent, and
  $\partial^+ A_{n-1}\subset \partial^- A_{n}$.
  One component of $\partial^- A_{n}$ is exactly $\partial^+ A_{n-1}$.
  $a_n\ge 2$. Therefore,
  there exists at least one other component of $\partial^- A_{n}$,
  that does not belong to $A_{n-1}$.
  Therefore, $f\colon A_{n}\to A_{n-1}$ is at least $n+1$-\hspace{0em}sheeted covering.
  $A_{n}$ is of type $(1,*)$.
  Since $A_n$ is of type $(1, *)$, the Riemann-Hurwitz formula implies the existence of
  a singular point in $A_n$.

  By induction, we conclude that each $A_n$ for $n \geq 1$ contains a singular point, proving the lemma.
\end{proof}

Note that the $n$-th step establishes that $f$ is at least an $n$-sheeted covering.
This implies that not only $f$ has infinitely many singular points,
but also $f$ is not a finite-to-one map.
As we consider $f$ on the compact manifold $M$, both those cases are forbidden.

\begin{corollary}\label{cor:backward_chain_split_homeo}
  Let the Pseudo-B\"{o}ttcher component $S$ contain singular points.
  If $\{A_m\}_{m \geq 0}$ is a main backward infinite chain,
  then there exists an $N > 0$ such that for all $n \geq N$:
  \begin{enumerate}
  \item Atoms $A_n$ are of type $(*, a_n)$ with $a_n > 1$.
  \item The restriction of $f$ to $A_n$ is a homeomorphism.
  \end{enumerate}
\end{corollary}

\subsection{Properties of Pseudo-B\"{o}ttcher Components' Partitioning.}

The concept of partitioning a Pseudo-B\"{o}ttcher component $S$ into
compact surfaces with boundary, termed atoms, is introduced in
Sections \ref{sec:PartitioningV0} and \ref{sec:atoms}.

Briefly, to construct a partitioning, we select a base annulus
$Q_c$, $Q_c\subset \overline{V_c} \subset V_0$, where
$Q_c:=\tau_0^{-1}\left(\left[c-1,c\right]\right)$,
$V_c=\tau_0^{-1}((-\infty,c))$,
$\overline{V_c}=\tau_0^{-1}((-\infty,c])$
for some $c\in (-\infty,0)$,
and $\tau_0$ is defined in Definition~\ref{def:pseudobottcher_basin1}.
Atoms are then defined as the connected components of $f^{-l}(f^m(Q_c))$ for $l, m \geq 0$,
leading to the representation $S = \bigcup_{l, m \geq 0} f^{-l}(f^m(Q_c))$
(Formula \ref{f:S_fundII_lmpart}).

Let us study how properties of this partitioning are related to the topology of $S$.

\begin{lemma}\label{lm:S_eq_Pn_seq}
  There exists a sequence of nested connected molecules $\{P_{c+n}\}_{n \geq 0}$ such that $\bigcup_{n \geq 0} P_{c+n} = S$.
\end{lemma}

\begin{proof}
  $\overline{V_c}$ is a molecule since
  $\overline{V_c} = \bigcup_{n \geq 0} Q_{c-n} = \bigcup_{n \geq 0} f^n(Q_c)$.

  We construct a sequence $\{P_{c+n}\}_{n \geq 0}$ inductively.
  The base case consists of the molecules $P_{c+0} := \overline{V_c}$
  and $P_{c+(-1)} := \overline{V_{c-1}}$.
  Assuming that for some $n$, we have a molecule $P_{c+n}$ satisfying $f(P_{c+n}) = P_{c+n-1}$,
  $\partial P_{c+n} \neq \emptyset$, and $\partial P_{c+n} \subset \tau^{-1}(c+n)$,
  we proceed to the induction step. Let $P_{c+n+1}$ be the connected
  component of $\tau^{-1}([-\infty, c+n+1])$ containing $P_{c+n}$.
  By construction, $\partial P_{c+n+1}\neq \emptyset$ and
  $\partial P_{c+n+1} \subset \tau^{-1}(c+n+1)$.
  From the connectivity of $P_{c+n+1}$ and the fact that $f(P_{c+n}) = P_{c+n-1}$,
  it follows that $f(P_{c+n+1}) = P_{c+n}$.
  Thus, the induction hypothesis holds for $n+1$.

  By induction, we obtain a sequence of sets $\{P_{c+n}\}_{n \geq 0}$.
  Their union, $K = \bigcup_{n \geq 0} P_{c+n}$, is a connected
  component of $S$ without boundary. Since $S$ is connected,
  we conclude that $K = S$, completing the proof.
\end{proof}

\begin{lemma}\label{lm:0critpoint_component_topology}
  If $S$ contains no singular points of $f$, then $S$ is homeomorphic to an open annulus.
\end{lemma}

\begin{proof}
  Consider the sequence $\{P_{c+n}\}_{n \geq 0}$ from Lemma~\ref{lm:S_eq_Pn_seq}.
  The molecule $P_{c+0} := \overline{V_c}$ serves as the base case for our induction.

  Assume that for some $n$
  the molecule $P_{c+n}$ from Lemma \ref{lm:S_eq_Pn_seq}
  is homeomorphic to a closed disc with one punctured point.
  In particular, the boundary of $P_{c+n}$ is a circle and its
  internals $\operatorname{Int}(P_{c+n})$ is an open annulus.
  Consider the function $\tau$ from Lemma~\ref{lm:pseudo_Butcher_timeline_coordinate}
  and let $\tau_{c+n}: \operatorname{Int}(P_{c+n}) \to (-\infty, 0)$ be the
  restriction of the function $\tau + c + n$ to $\operatorname{Int}(P_{c+n})$.
  Further assume that
  $\operatorname{Int}(P_{c+n})$ with the function $\tau_{c+n}$ satisfies the
  conditions of Definition \ref{def:pseudobottcher_basin1}.

  For the induction step, since $\partial P_{c+n}$ is a circle,
  there exists a unique atom $A_{c+n+1}$ adjacent to $P_{c+n}$.
  Since $S$ contains no singular points of $f$,
  Lemma~\ref{lm:atom_planar} implies that $A_{c+n+1}$ is homeomorphic to a closed annulus
  and each $\partial^+ A$ and $\partial^- A$ are homeomorphic to a circle.
  The union of $P_{c+n}$ and $A_{c+n+1}$ is the connected component
  of $\tau^{-1}([-\infty, c+n+1])$ containing $P_{c+n}$, and thus
  equals $P_{c+n+1}$ by definition.
  By construction, $P_{c+n+1}$ is homeomorphic to $P_{c+n}$ and to a
  punctured closed disc,
  $\partial P_{c+n+1}$ is homeomorphic to a circle,
  and $\operatorname{Int}(P_{c+n+1})$ is homeomorphic to an open annulus.

  Take the function $\tau$ from Lemma~\ref{lm:pseudo_Butcher_timeline_coordinate}
  and let $\tau_{c+n+1}: \operatorname{Int}(P_{c+n+1}) \to (-\infty, 0)$ be the
  restriction of $\tau + c + n + 1$ to $\operatorname{Int}(P_{c+n+1})$.
  By Lemma \ref{lm:atom_trivial_fiber_bundle}, each level curve of $\tau_{c+n+1}$ is a circle.
  The function $\tau$ itself satisfies the conditions of
  Lemma~\ref{lm:pseudo_Butcher_timeline_coordinate}.
  Thus, $\operatorname{Int}(P_{c+n+1})$ with the function $\tau_{c+n+1}$ satisfies the
  conditions of Definition \ref{def:pseudobottcher_basin1}.

  By induction, we obtain a sequence of sets $\{P_{c+n}\}_{n \geq 0}$
  whose union $K = \bigcup_{n \geq 0} P_{c+n}$ is homeomorphic to an open annulus
  and is a connected component of $S$.
  Since $S$ is connected, we conclude that $K = S$, proving the lemma.
\end{proof}

\begin{corollary}\label{cor:stop_0critpoint}
  If $S$ has singular points of $f$, then
  in the sequence $\{P_{c+n}\}_{n \geq 0}$ from Lemma~\ref{lm:S_eq_Pn_seq},
  there exists an $n\geq 1$ such that
  the induction hypothesis of Lemma~\ref{lm:S_eq_Pn_seq} holds for
  $P_0$, \dots, $P_{c+n-1}$,
  but the adjacent atom $A^0_{c+n}$ contains singular points,
  preventing the continuation of the induction process for $P_{c+n}$.
\end{corollary}

\begin{proof}
  If no adjacent atom $A^0_{c+k}$ contains singular points for any $k$,
  then the process described in the proof of Lemma~\ref{lm:0critpoint_component_topology}
  can be continued indefinitely, resulting in $S$ without singular points.
\end{proof}

Define $V_{c+n}:=\operatorname{Int}(P_{c+n})$
for each $n\geq 0$ such that $P_{c+n}$ satisfies the induction hypothesis of Lemma~\ref{lm:S_eq_Pn_seq}.
Clearly, $\overline{V_{c+n}}=P_{c+n}$.
The sets $V_{c+n}$ can be viewed as extensions of the structure in
Definition \ref{def:pseudobottcher_basin1} from the neighborhood $V_0$
further into $S$.

Recall that $X^\perp$ denotes the neutral saturation of the set
$X$ (Definition~\ref{def:neutral_saturation}).

\begin{definition}
  Let $S$ has singular points.
  Let the induction hypothesis of Lemma~\ref{lm:S_eq_Pn_seq} holds for
  $P_0$, \dots, $P_{c+n-1}$,
  but the adjacent atom $A^0_{c+n}$ contains singular points
  (see Corollary~\ref{cor:stop_0critpoint}).
  Then:
  \begin{itemize}
  \item The molecule $P_{c+n-1} = \overline{V_{c+n-1}}$ is called the \emph{main trunk}.
  \item The atom $A^0_{c+n}$ containing singular points is called the \emph{main stump}.
  \item The molecule $RT$ composed of the main stump and all atoms connected to it via forward chains (Definition \ref{def:forward_chain_of_atoms}) is called the \emph{main root}.
  \item
    The connected components (if any) of $\overline{V_{c+n-1}}^\perp \setminus \overline{V_{c+n-1}}$
    are called \emph{auxiliary trunks}.
  \item
    The connected components (if any) of $(A^0_{c+n})^\perp \setminus A^0_{c+n}$
    are called \emph{auxiliary stumps}.
  \item The connected components (if any) of $RT^\perp \setminus RT$ are called \emph{auxiliary roots}.
  \end{itemize}
\end{definition}

By definition, the main trunk, stump and
root are non-empty, while auxiliary trunks, stumps and roots can be empty sets.
Note that the decomposition of $S$ into main and auxiliary trunks, stumps and roots
is defined only when $S$ contains singular points.

\begin{figure}[htbp]\label{fig:PartitioningSexamples}
    \includegraphics[scale=0.7]{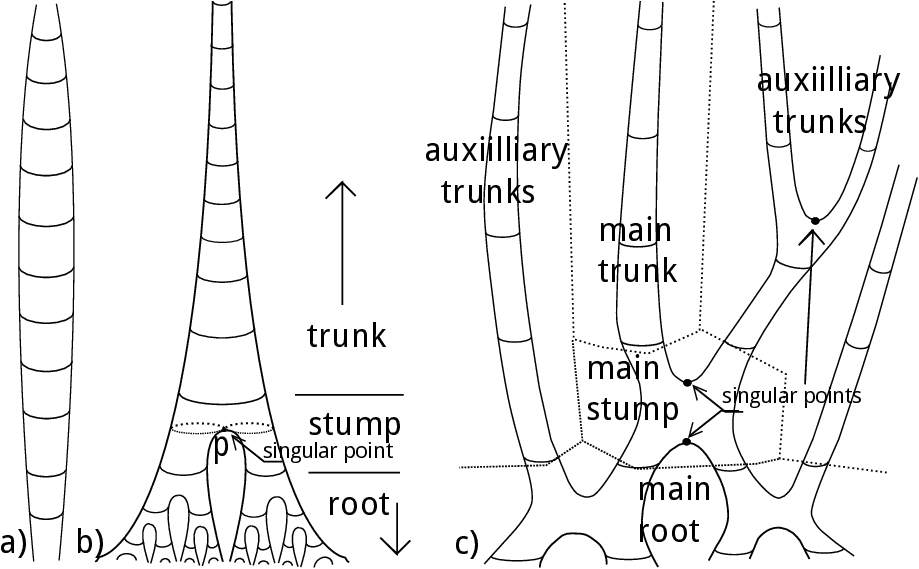}
\caption{
  Pseudo-B\"{o}ttcher components with partitioning.
}
\end{figure}

Figure~\ref{fig:PartitioningSexamples} illustrates examples of partitioned Pseudo-B\"{o}ttcher components:
\begin{itemize}[leftmargin=*]
  \item [a)]
    $S$ without singular points. All atoms are annuli.
  \item [b)]
    $S$ with a single singular point. All atoms are main. The first non-annulus atom is the
stump that marks the end of the trunk and the beginning of the root.
  \item [c)]
    $S$ with multiple singular points.
    The main stamp has two. An auxiliary trunk has one. Note that further
    down, outside the depicted region,
    $S$ may contain additional auxiliary trunks, roots, and singular points.
\end{itemize}

\begin{lemma}\label{lm:main_are_main_auxiliary_are_auxiliary}
  Atoms in the main trunk, stump, and root are main, while atoms in auxiliary trunks, stumps, and roots are auxiliary.
\end{lemma}

\begin{proof}
  It suffices to show that an atom is main if and only if it belongs
  to the main trunk or root (recall that the main stump is part of the
  main root by definition).

  The main trunk $\overline{V_{c+n-1}}$ consists of main atoms due to
  the structure of Definition~\ref{def:pseudobottcher_basin1}.
  The main stump is a main atom by
  Lemma~\ref{lm:main_and_auxiliary_atoms}-\ref{enum:main_and_auxiliary_atoms_1}.
  By Corollary~\ref{cor:backward_chain_is_main}, all atoms in the main
  root are main.

  Consider a main atom $A$ that is not the main stamp and not part of the main trunk.
  The sequence $A, f(A), f(f(A)), ...$ forms a forward chain since $A$ is main.
  As the neighborhood $V_0$ from Definition~\ref{def:pseudobottcher_basin1}
  is restriction of a strictly attracting neighborhood on $S$,
  there exists $M \geq 2$ such that $f^m(A) \in V_0$ for all $m \geq M$.
  In other words, $f^m(A)$ for all $m > M$ belong to the main trunk.
  Since the main stump disconnects the main trunk from the rest of the component $S$,
  the molecule $C = \bigcup_{i \geq 0} f^i(A)$ must contain the main stump.
  Consequently, a portion of $C$ forms a forward chain connecting $A$ to
  the main stump. Then, by definition, $A$ belongs to the main root.
\end{proof}

By definition, any atom $A$ in the main root $RT$ is a preimage of the main stump $A^0_{c+n}$,
i.e., there exists $m \geq 0$ such that $A \subset f^{-m}(A^0_{c+n})$.

\begin{lemma}\label{lm:auxiliary_trunk_root_corellation}
  $S$ has auxiliary trunks if and only if $S$ has auxiliary roots.
\end{lemma}

\begin{proof}
  Suppose $S$ has no auxiliary trunks.
  Then any semi-infinite forward chain eventually enter the main trunk
  and, by definition, all atoms in $S$ are main.

  Suppose $S$ has auxiliary trunks but no auxiliary roots.
  Since any atom of type $(a>1,b>1)$, connected to an auxiliary trunk,
  creates an auxiliary root, the atoms of $S$ should be of type $(1,*)$.
  But this case is prohibited by Lemma~\ref{lm:Y_seq_is_infinite}.
\end{proof}

\begin{lemma}\label{lm:atom_with_1crit_point}
  If the main stump $A^0_{c+n}$ has only one singular point $p$ with defect $n_p-1$,
  then $A^0_{c+n}$ is either of type $(n_p,1)$ or $(1,n_p)$, but not both.
\end{lemma}

\begin{proof}
  By Stoilov Lemma (\ref{lm:stoilov_lemma}),
  the map $f$ is equivalent to the holomorphic function $z^{n_p}$ at the point $p$,
  up to homeomorphisms in neighborhoods of $p$ and $f(p)$.
  In particular, the regular level curve of the function $\tau$
  passing through $f(p)$ in a neighborhood of $f(p)$
  has $n_p$ preimages in the neighborhood of $p$ that intersect at the point $p$.
  Therefore, $p$ is also a singular point of the function $\tau$ and a
  singular point of the level curve $A^0_{c+n}\cap \tau^{-1}\tau((p))$.
  Topologically, there is only one singular level curve with one
  singular point: a wedge of circles.

  By the choice of $c$ in Section~\ref{sec:atoms}, $c+n-1 < \tau(p) < c+n$.
  Consider the partitioning on the level curves of the function $\tau$ in
  the atom $A^0_{c+n}$.
  $f$ takes $A^0_{c+n}$ onto $Q_{c+n-1}$. Any level curve of $\tau$ in $Q_{c+n-1}$ is a circle.
  $f$ is a local homeomorphism in $A^0_{c+n}$ except for $p$.
  Therefore, in the atom $A^0_{c+n}$ any level set of $\tau$ consists of finite number of circles
  except for the singular level set $\tau^{-1}(\tau(p))$.
  Denote by $\gamma_1$ the connected component of the singular level set
  $\tau^{-1}(\tau(p))$ that belongs to $A^0_{c+n}$.
  $f$ restricted to $\gamma_1$ is a branched covering of the circle
  $\tau^{-1}(\tau(p)+1) \cap Q_{c+n-1}$ by $\gamma_1$ with the branch
  point $p$. Therefore, $\gamma_1$ is a wedge of $n_p$ circles.
  \begin{figure}[htbp]
    \centering
    \includegraphics[scale=0.2]{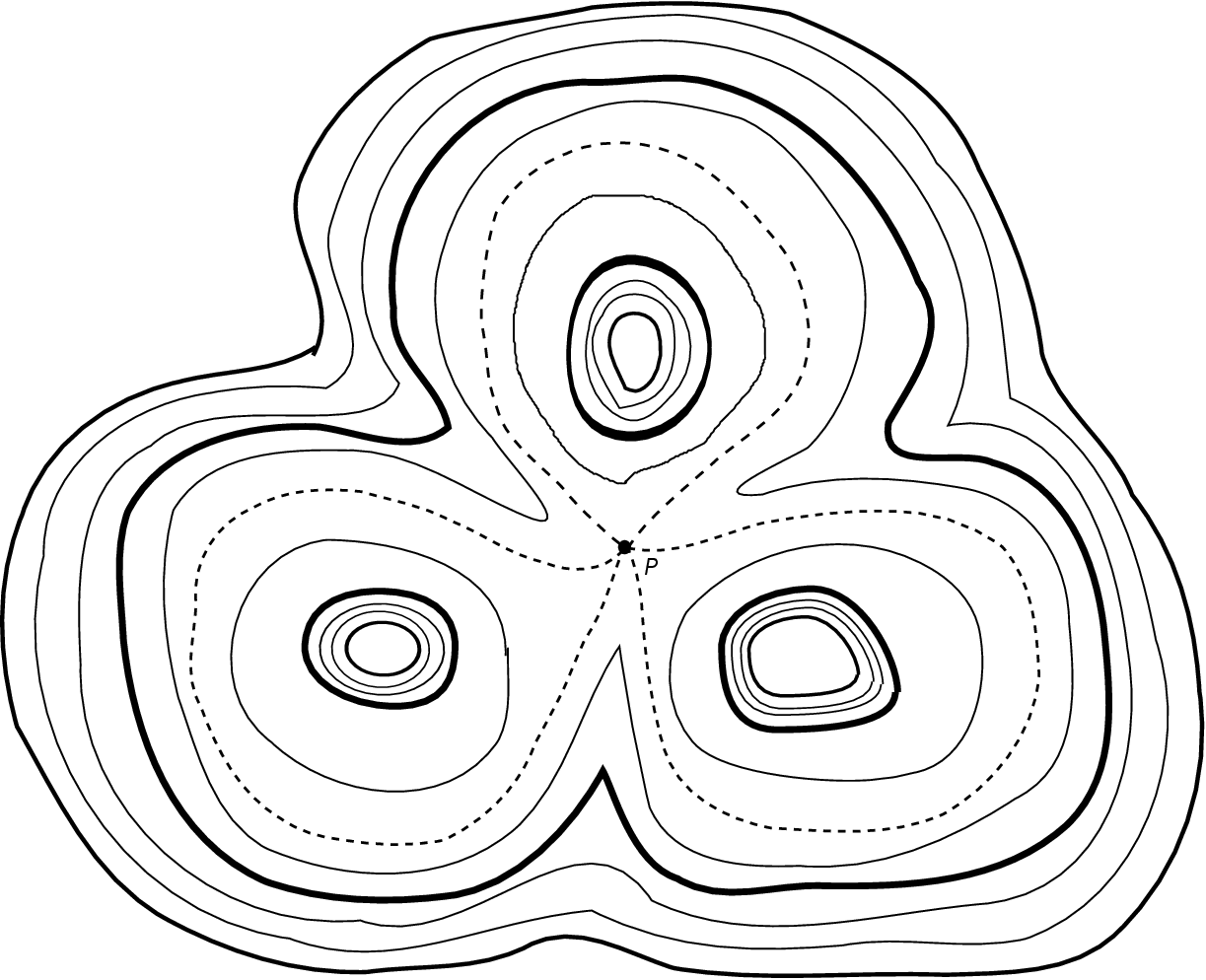}
    \caption{Wedge of three circles in $A^0_{c+n}$ for $n_p$=3.}
    \label{fig:fig3}
  \end{figure}
  As an atom, $A^0_{c+n}$ is a planar surface with boundary
  $\partial A^0_{c+n}$, which is divided by construction into
  $\partial^- A^0_{c+n}$ and $\partial^+ A^0_{c+n}$.
  $\gamma_1$ is embedded in $A^0_{c+n}$. The wedge of $n_p$ circles,
  $\gamma_1$ divide $A_{c+n}$ and, correspondingly, $\partial A^0_{c+n}$
  into $n_p+1$ connected components such that
  $n_p$ components of the boundary belong to one level set of the
  function $\tau$ and one component belongs to another
  level set of $\tau$. In other words, $A^0_{c+n}$ is either of type
  $(1,n_p)$ or of type $(n_p, 1)$.

  As both cases are mutually exclusive, that completes the proof.
\end{proof}

\begin{lemma}\label{lm:1crit_point_atom_type}
  if $S$ has only one singular point of $f$  with defect $n_p-1$
  then $n_p=|\deg(f)|$, and $S$ consists of the main trunk
  and the main root, where the atoms of main root all have type $(1,n_p)$.
\end{lemma}

\begin{proof}
  By Corollary~\ref{cor:stop_0critpoint},
  the singular point of $S$ is located within
  the main stamp $A^0_{c+n}$.
  By Lemma~\ref{lm:atom_with_1crit_point},
  $A^0_{c+n}$ is either of type $(n_p,1)$ or $(1,n_p)$, but not both.
  According to Lemma~\ref{lm:atom_ab_type_map_classification}, since $n_p>1$,
  all atoms in $f^{-m}(A^0_{c+n})$ for $m\ge 1$ are homeomorphic to $A^0_{c+n}$
  in the absence of other singular points, and thus have the same type.
  If $A^0_{c+n}$ is of type $(n_p,1)$, we obtain the
  situation of Lemma~\ref{lm:Y_seq_is_infinite}, which is forbidden.
  Therefore, $A^0_{c+n}$ is of type $(1,n_p)$,
  and all atoms of the main trunk are also of type $(1,n_p)$.

  Since $f$ is an inner mapping of manifolds, $\deg(f)$ is constant
  and equals the number of preimages in each point
  of $M\setminus \Sing(f)$ \cite{Trokhimchuk2008__en}.
  As $S$ is a totally invariant set,
  the number of local preimages at any point in $S$ equals
  the global number of preimages, implying $n_p=|\deg(f)|$.
\end{proof}

Note that components of strict wandering with a single singular point
permitted in Lemma~\ref{lm:1crit_point_atom_type} do exist.
Basin of attraction toward infinity in the holomorphic map $z^{n_p}+1$
is the example. 

\section{Topology of pseudo-B\"{o}ttcher components of strict wandering}
\label{sec:pseudo_bottcher_topology}

Let $S$ be a pseudo-B\"{o}ttcher component of strict wandering.
Since $S$ is an open subset of the compact 
surface $M$, then
$S$ is non-compact surface without boundary.
Let's find the topological type of the surface $S$.


Non-compact surfaces are classified in papers
\cite{Richards63_zbMATH03251274,MishchenkoPryshlyak2004_zbMATH02169042,
MishchenkoPryshlyak2007_zbMATH05158692,MishchenkoPryshlyak2009_zbMATH05697575}.

Let's recall the corresponding definitions following the paper~%
\cite{MishchenkoPryshlyak2007_zbMATH05158692}.

\begin{definition}
  A nested set sequence $A_1 \supset A_2 \supset A_3 \supset \dots$
  \emph{majorizes} another nested set sequence
  $B_1 \supset B_2 \supset B_3 \supset \dots$ if
  for every set $A_n$ in the first sequence, there exists a set $B_N$
  in the second sequence such that $B_N \subset A_n$.
\end{definition}

\begin{definition}
  Let $K_1 \supset K_2 \supset K_3 \supset \dots$ be a nested sequence of
  connected unbounded regions in a surface $S$ such that:
  \begin{enumerate}[leftmargin=2em, label={\rm\alph*)}]
  \item[a)]
    the boundary of $K_n$ in $S$ is compact for all $n$;
  \item[b)]
    for any bounded subset $A$ of $S$,
    there exists a sufficiently large $n$ such that $K_n \cap A = \emptyset$.
  \end{enumerate}
  Then $K_1 \supset K_2 \supset K_3 \supset \dots$ is called an \emph{end-representing sequence}
  (or \emph{end representation}).
\end{definition}

\begin{definition}
  Two end-representing sequences $K_1 \supset K_2 \supset K_3 \supset \dots$
  and $Q_1 \supset Q_2 \supset Q_3 \supset \dots$ are equivalent if
  each sequence majorize the other.
\end{definition}

\begin{definition}
The equivalence class of end-representing sequences containing $k = K_1 \supset K_2 \supset K_3 \supset \dots$ is
called an end and denoted by $k^*$.
\end{definition}

\begin{definition}
  Let $k^*$, represented by $k = K_1 \supset K_2 \supset K_3 \supset \dots$ be an end of $S$.
  We say that $k^*$ is planar and$\backslash$or orientable if the sets
  $K_n$ are planar and$\backslash$or orientable for all sufficiently large $n$.
\end{definition}

\begin{definition}
  The ideal boundary (the set of ends) $B(S)$ of a surface $S$ is a topological
  space having the ends of $S$ as elements and endowed with the following topology: for any
  set $U$ in $S$ whose boundary in $S$ is compact, we define $U^*$ to be the set of all ends $k^*$,
  represented by some $k = K_1 \supset K_2 \supset K_3 \supset \dots$
  such that $K_n \subset U$ for a sufficiently large $n$; we take the
  set of all such $U^*$ to be a basis for the topology of $B(S)$.
\end{definition}

A triplet $B(S) \supset B'(S) \supset B''(S)$, where $B(S)$ is the
whole ideal boundary, $B'(S)$ is the part which is not planar,
and $B''(S)$ is the part which is not orientable, is a topological
invariant of $S$ up to a compact subsurface.

From Lemma~\ref{lm:finite_molecule_planar} and the fact that
$S$ is embedded in the compact surface $M$,
$S$ is a planar surface and $B'(S) = B''(S) = \emptyset$.
Then,
according to the paper~\cite{Richards63_zbMATH03251274},
the topological type of $S$ is defined 
by the topology of the set $B(S)$.

Using the function $\tau\colon S\to \rr$ from Lemma~\ref{lm:pseudo_Butcher_timeline_coordinate}.
define the sets $\alpha_n:=\tau^{-1}((-\infty,-n))$,
$B_n:=\tau^{-1}([-n,n])$, $\omega_n:=\tau^{-1}((n,\infty))$.
Intuitively, each $\omega_n$ can be seen as
a neighborhood of the repeller $R$ in $S$, while $\alpha_n$ is a
neighborhood of the set $A^\perp$ (the neutral saturation of the attractor
$A$) within $S$.
By definition, $S=\alpha_n\cup B_n \cup \omega_n$ for each $n> 0$.

The sequence $(B_n)_{n > 0}$ consists of bounded sets in $S$, while
$(\alpha_n)_{n > 0}$ and $(\omega_n)_{n > 0}$ are its complimentary
sequences.
Although it's possible for
$(\alpha_n)_{n>0}$ and $(\omega_n)_{n>0}$ to be end-representing sequences
if the sets $\alpha_n$ and $\omega_n$ are connected for all $n >
0$, this is generally not the case.
However, by construction,
for any representation 
$k = K_1 \supset K_2 \supset K_3 \supset \dots$
of any end $k^*$ of $S$,
the sequence
$(K_n)_{n>0}$ is majorized by either $(\alpha_n)_{n>0}$ or $(\omega_n)_{n>0}$.

\begin{definition}
  If representations of an end $k^*$ are majorized by the sequence $(\alpha_n)$,
  we say that $k^*$ belongs to the set $AIB$ (Attractor{-}side Ideal Boundary).
\end{definition}

\begin{definition}
  If representations of an end $k^*$ are majorized by the sequence $(\omega_n)$,
  we say that $k^*$ belongs to the set $RIB$ (Repeller{-}side Ideal Boundary).
\end{definition}

By definition,
the ideal boundary $B(S)$ is the union of $AIB$ and $RIB$.

Both sets
$AIB$ and $RIB$ are non-empty, as we can construct an end from
each sequence $(\alpha_n)$ and $(\omega_n)$ by arbitrarily choosing
a connected component from $\alpha_n$ and $\omega_n$ for each $n> 0$.
In particular, $AIB$ always contains an end $\alpha^*$ generated by the sequence
$V_0 \supset V_{-1} \supset ... \supset V_{-n} \supset ...$,
where $V_i=\tau_0^{-1}(-\infty,i))$ are neighborhoods of the attractor $A$.

The set $V_0$ is homeomorphic to an open annulus, implying that the end $\alpha^*$
is an isolated point in the ideal boundary $B(S)$.

From Lemma~\ref{lm:one_adjacent_boundary}, if $S$ has no singular points of $f$
then $B(S)$ consists of two ends, one from $AIB$ and the other from $RIB$.

\begin{lemma}\label{lm:atoms_yuild_ends_infinity}
  If $S$ has an atom of type $(k,*)$, $k>1$, then $AIB$ is a zero-dimensional infinite set.
  If $S$ has an atom of type $(*,k)$, $k>1$, then $RIB$ is a zero-dimensional infinite set.
\end{lemma}

\begin{proof}
  By definition,
  any open unbounded connected subset of $S$ is a neighborhood in $B(S)$
  and contains at least one point in $B(S)$.

  Consider the oriented Conrod-Reeb graph $\Gamma$ of the function
  $\tau$ (a factor space of $S$ by connected components of $\tau$'s
  level curves).
  According to Lemma~\ref{lm:finite_molecule_planar}, $\Gamma$ is a graph without
  cycles, a tree.
  An atom of type $(p,q)$ corresponds to a bounded area of $\Gamma$ with $q$
  incoming edges and $p$ outgoing edges.
  Ignore atoms of type $(1,1)$ as they correspond to simple
  segments of $\Gamma$.

  An atom of type $(p,q)$ divides $\Gamma$ into
  $p+q$ subgraphs that correspond to open unbounded connected subsets of $S$.
  Of these subgraphs, $q$ are mapped by $\tau$ to intervals of the form $[const,\infty)$,
  corresponding to non-empty neighborhoods in $RIB$ and hence at
  least $q$ ends in $RIB$.
  Similarly, $p$ subgraphs are mapped by $\tau$ to intervals of the form $(-\infty, const]$,
  corresponding to non-empty neighborhoods in $AIB$ and at
  least $p$ ends in $AIB$.

  If there are $k$ atoms of type $(p,*)$, there are at least $k(p-1)$
  distinct ends in $AIB$.
  Similarly, for $k$ atoms of type $(*,q)$, there are at least $k(q-1)$
  distinct ends in $RIB$.
  According to Lemma~\ref{lm:atom_ab_type_map_classification},
  if there is an atom $A$ of type $(p,*)$, $p>1$, then each preimage of $A$ is
  an atom of type $(p',*)$, where $p'\ge p$.
  Analogously, for an atom $A$ of type $(*,q)$, $q>1$, each preimage of $A$ is
  an atom of type $(*,q')$, where $q'\ge q$.

  Since every atom has infinitely many preimages in $S$, there are infinitely
  many ends in $AIB$ if $p>1$ and in $RIB$ if $q>1$.

  A topological space is zero-dimensional if every open cover 
  has a refinement which is a cover by disjoint open sets.
  For $AIB$ and $RIB$ this refinement is provided by connected components
  of $(\alpha_n)$ and $(\omega_n)$ correspondingly.
  This completes the proof.
\end{proof}

As a consequence of Lemmas~\ref{lm:1crit_point_atom_type},
\ref{lm:atoms_yuild_ends_infinity},
we got

\begin{lemma}\label{lm:1crit_point_component_topology}
  If $S$ has only one singular point of $f$
  then $|AIB|=1$ and $|RIB|=\infty$.
\end{lemma}

\begin{lemma}\label{lm:AIB_ends_are_isolated}
  Any end that belongs to $AIB$ is an isolated point in $B(S)$.
\end{lemma}

\begin{proof}
  Select $N\in\zz$ such that $N+c<\mathrm{min}_{p\in\Sing(f)}(\tau(p))$.
  Consider a connected component $T$ of the set $\tau^{-1}((-\infty,N+c))$.
  Either $T$ is $V_{N+c}$ or there exists $m\in\nn$ such that $f^m(M)=V_{N-m+c}$,
  since $V_{N+c}$ is restriction of a strictly attracting neighborhood of the attractor.

  In the first case, $V_{N+c}$ is homeomorphic to an open annulus by definition.
  In the second case, by the choice of $N$, $T$ contains no
  singular points of $f$ or $f^m$, and $f^m$ is a regular covering
  of $V_{N-m+c}$ by $T$.
  Thus, by Riemann-Hurvitz formula, $T$ is also an open annulus.

  As an unbounded connected set and an open annulus, each $T$
  represents both a single end in $AIB$ and its open neighborhood in $B(S)$
  containing no other ends of $S$. This proves the lemma.
\end{proof}


After choosing the base annulus $Q_c$ and fixing the partitioning of $S$ into atoms,
it is convenient to use end-representing sequences
consistent with the partitioning of $S$.

Since we can ignore any finite part or finite subsequence in these
semi-infinite sequences, for convenience, we align the sequences
to start with generation 0 atoms (Definition~\ref{def:atom_comp_tau}).

\begin{definition}
  A representation $W_0 \supset W_1 \supset W_2 \supset \dots$ of an end $w^*\in AIB$
  is called \emph{canonical}
  if, $\forall i\ge 0$,
  $W_i$ is a connected component of the set
  $\tau^{-1}\left(\left(-\infty,c-i\right]\right)$.
\end{definition}

\begin{definition}
  A representation $W_0 \supset W_1 \supset W_2 \supset \dots$ of an end $w^*\in RIB$
  is called \emph{canonical}
  if, $\forall i\ge 0$,
  $W_i$ is a connected component of the set
  $\tau^{-1}\left(\left[c-1+i,\infty\right)\right)$.
\end{definition}

By construction, each set $W_i$ is a molecule.
For any end $k^*$ its
canonical representation is uniquely determined
and depend only on the choice of the constant $c$ in partitioning and
the end itself.

Canonical representations of the ends in $AIB$ are simple.
In the canonical representation of the attractor's end $\alpha^*\in AIB$,
$\overline{V_c} \supset \overline{V_{c-1}} \supset ... \supset \overline{V_{c-n}} \supset ...$,
each $\overline{V_{c-i}}=\tau_0^{-1}(-\infty,c-i])$ is main forward semi-infinite chain.
Any other end in $AIB$ has canonical representation consisting of
preimages of subchains of the main forward semi-infinite chain $\overline{V_c}$.

For the ends in $RIB$
the only problem with canonical representations is topological
complexity of their molecules $W_i$.

\begin{definition}
A semi-infinite backward chain $C=\bigcup_{i\geq 0}A_i$ such that $\tau(C)=[c-1,\infty)$
is called \emph{canonical}.
\end{definition}

Each canonical backward chain $(A_i)_{i\geq 0}$
naturally define a
canonical representation of an end $w^*\in RIB$
$W_0 \supset W_1 \supset W_2 \supset \dots$, where
$W_i$ is a connected component of the set
$\tau^{-1}\left(\left[c+i,\infty\right)\right)$ containing $A_i$.

Conversely, if we have a canonical representation
$w = W_0 \supset W_1 \supset W_2 \supset \dots$
of an end $w^*\in RIB$, then
each $W_i$ also contains at least one semi-infinite backward chain.
From connectivity, one can expect to have a canonical semi-infinite backward
chain such that each $W_i$ contains its subchain.

\begin{definition}
  A canonical representation $W_0 \supset W_1 \supset W_2 \supset \dots$ of an end $w^*\in RIB$
  contains a canonical backward chain $C=\bigcup_{i\geq 0}A_i$
  if for all $i\geq 0$ $W_i\cap C =\bigcup_{j\geq i}A_j$.
\end{definition}

\begin{lemma}\label{lm:canonical_sequence_exists_chain}
  A canonical representation $W_0 \supset W_1 \supset W_2 \supset \dots$ of an end $w^*\in RIB$
  always contains at least one canonical backward chain $C=\bigcup_{i\geq 0}A_i$.
\end{lemma}

\begin{proof}
  Suppose all atoms in $W_0$ are of type $(1,1)$.
  Then $W_0$ itself is the required chain $C$.

  Suppose all atoms in $W_0$ are of type $(*,1)$.
  Start the chain $C$ by arbitrarily choosing an atom $A_0$ from $W_0 \setminus W_1$.
  Since $A_0$ is of type $(*,1)$, there is unique atom $A_1$ connected to $\partial^+ A_0$.
  Continuing this process infinitely, we obtain $C$.

  Suppose all atoms in $W_0$ are of type $(1,*)$.
  For any $i\geq 0$, the set $W_{i} \setminus \tau^{-1}\left(\left[c+i+1,\infty\right)\right)$ consists of a
  single atom $A_i$.
  Together, the atoms $A_i$ form the required chain $C$.

  Consider the general case.
  For all $i\geq 0$ and $j>0$, denote by $R(i,j)$ the set of atoms from $W_i \setminus W_{i+1}$
  reachable from $W_{i+j}$ by a forward chain.
  By arbitrarily choosing an atom in $W_{i+j}$ and arbitrarily continuing
  a forward chain of the length $j+1$, we will obtain an atom in $W_i$
  due to connectivity.
  Therefore, the sets $R(i,j)$ are not empty for all $i\geq 0$ and $j>0$.
  By definition, $R(i,j+1)\subseteq R(i,j)$.
  Let $R(i):=\bigcap_{j>0} R(i,j)$.
  For all $i\geq 0$ the sets $R(i)$ are not empty as intersections of non-empty nested sets.


  Any atom of $W_i$ adjacent to an atom from $R(i+1)$ belongs to
  $R(i)$ by definition as they form a chain. Therefore,
  any atom of $R(i)$ has an adjacent atom that belongs to $R(i+1)$.

  Start the chain $C$ by arbitrarily choosing an atom $A_0$ from $R(0)$.
  Inductively, given an atom $A_i\in R(i)$, 
  there exists an atom $A_{i+1}\in R(i+1)$ connected to $\partial^+ A_i$.
  Continuing this process infinitely, we obtain $C$.
\end{proof}

\begin{lemma}\label{lm:canonical_sequence_two_chains_stabilize}
  If a canonical representation $W_0 \supset W_1 \supset W_2 \supset \dots$ of an end $w^*\in RIB$
  contains two canonical backward chains $C_1=\bigcup_{i\geq 0}A^1_i$
  and $C_2=\bigcup_{i\geq 0}A^2_i$, then there exists $J\geq 0$ such
  that their subchains $C^J_1=\bigcup_{i\geq J}A^1_i$ and $C^J_2=\bigcup_{i\geq J}A^2_i$
  coincide: $C^J_1=C^J_2$.
\end{lemma}

\begin{proof}
  Suppose there exists $i\geq 0$ such that subchains $C^i_1:=\bigcup_{j\geq i}A^1_j$ and
  $C^i_2:=\bigcup_{j\geq i}A^2_j$ do not intersect.
  $W_i$ is connected set which contains both $C^i_1$ and $C^i_2$.
  Therefore, there exists a bridge $B$ in $W_i$
  that joins $C^i_1$ and $C^i_2$.
  Since by construction the bridge $B$ is a finite molecule (connected set of atoms),
  there exists $k>i$ such that $W_k\cap B=\emptyset$.

  By Corollary~\ref{cor:atom2disconnect},
  any atom disconnects $S$ into multiple connected components.
  In particular, atoms of $B$ disconnect $W_i$ into connected components
  such that there are separated connected components for each $C^i_1$ and $C^i_2$.
  By definition of canonical representation, $W_k$ is a connected component.
  Therefore, only one of the subchains $C^i_1$ and $C^i_2$ can be contained in $W_k$.
  We got a contradiction.

  As a result, the chains $C_1$ and $C_2$ intersect.
  Suppose, after intersection they do not coincide, but eventually diverge.
  Then there are two possibilities: either they no more intersect,
  or they later intersect again. First possibility leads to
  contradiction as shown before.
  Second possibility creates a handle in the union of $C_1$ and $C_2$,
  which contradicts Lemma~\ref{lm:finite_molecule_planar}.
  Both possibilities end with contradiction, hence, $C_1$ and $C_2$
  eventually coincide.
\end{proof}

\begin{definition}\label{def:main_auxiliary_canonical_representation}
  A canonical representation $W_0 \supset W_1 \supset W_2 \supset \dots$ of an end $w^*\in RIB$
  is called \emph{main},
  if it contains a main canonical backward chain.
  Otherwise, it is called \emph{auxiliary}.
\end{definition}

\begin{lemma}\label{lm:main_canonical_sequence_to_chain}
  Any main canonical representation
  $w = W_0 \supset W_1 \supset W_2 \supset \dots$ of an end $w^*\in RIB$
  contains 
  only one main canonical backward chain $C=\bigcup_{i\geq 0}A_i$.
\end{lemma}

\begin{proof}
  Suppose $w$ contains two main canonical backward chains $C_1=\bigcup_{i\geq 0}A^1_i$
  and $C_2=\bigcup_{i\geq 0}A^2_i$.
  By Lemma~\ref{lm:canonical_sequence_two_chains_stabilize},
  there exists $J\geq 0$ such
  that their subchains $C^J_1=\bigcup_{i\geq J}A^1_i$ and $C^J_2=\bigcup_{i\geq J}A^2_i$
  coincide: $C^J_1=C^J_2$.
  In particular, $A:=A^1_J=A^2_J$ is their common main atom.
  By Lemma~\ref{lm:train_of_atoms0}, the only main forward chain of
  the length $J$ that starts in $A$ is the chain $(A$, $f(A)$, \dots, $f^J(A))$.
  That means that $C_1$ and $C_2$ coincide.
\end{proof}

Any main canonical representation $w$ contains a unique main canonical backward chain $C$.
Any other canonical backward chain, contained in $w$, is a mixed chain.
Let us define designated chains for auxiliary canonical
representations too.

\begin{definition}\label{def:premain_auxiliary_canonical_backward_chain}
  An auxiliary semi-infinite backward chain $C=\bigcup_{i\geq 0}A_i$ is
  called \emph{pre{-}main}, if there exists $n\geq 1$ such that
  for each $0 < i < n$
  $f^i(C)$ is an auxiliary semi-infinite backward chain, but
  $f^n(C)$ is a main semi-infinite backward chain.
\end{definition}

Note that any chain will eventually be mapped by $f$ to a main chain.
However, not any auxiliary chain is pre{-}main:
some auxiliary chains will be mapped to a mixed chain first.

\begin{lemma}\label{lm:auxiliary_canonical_sequence_to_chain}
  Any auxiliary canonical representation
  $w = W_0 \supset W_1 \supset W_2 \supset \dots$ of an end $w^*\in RIB$
  contains a finite number of pre{-}main auxiliary canonical backward chains.
\end{lemma}

\begin{proof}
  Since by construction $V_c$ is restriction of a strictly attracting
  neighborhood of the attractor $A$ on $S$, there exists $J > 0$
  such that $f^J(W_0)\cap V_c\neq \emptyset$.

  Consider the nested sequence $f^J(w) = f^J(W_J) \supset f^J(W_{J+1}) \supset f^J(W_{J+2}) \supset \dots$.
  It is canonical sequence, as generation of atoms in
  $f^J(W_J)\setminus f^J(W_{J+1})$ is equal $J-J=0$.
  It is main sequence, as any backward chain that starts in main trunk
  $V_c$ is main.
  By Lemma~\ref{lm:main_canonical_sequence_to_chain},
  the main canonical representation $f^J(w)$ contains a unique main
  chain $C_{main}$.
  Let $P$ be a connectivity component of the set $f^{-J}(C_{main})$,
  distinct from $C_{main}$.
  By construction, $P$ is molecule. $f^J$ maps $P$ to $C_{main}$.
  On the level of individual atoms, the action of the map $f^J$ is described in
  Lemma~\ref{lm:atom_ab_type_map_classification}.
  Either $P$ does not have singular points of $f^J$, then $P$ is also
  a chain, or $P$ have singular points. Then $P$ is not a chain, but
  forks in a finite number of atoms having singular points.

  We can ignore branches in $P$ that fork towards repeller. An atom with singular points
  disconnects $P$ into multiple connected components, but only one
  connected component can continue the chain that belongs to the
  canonical representation $w$.
  Each branch in $P$ that fork towards attractor create a separate
  backward chain that belongs to the canonical representation $w$.
  As the number of singular points is finite,
  there are finite number of chains that maps to $C_{main}$.
\end{proof}

\begin{definition}\label{lm:RIB_main_end}
  Let $k^*$ be an end in $RIB$ with the main canonical representation.
  Then $k^*$ is called a \emph{main} end.
\end{definition}

\begin{definition}\label{lm:RIB_auxiliary_end}
  Let $k^*$ be an end in $RIB$ with the auxiliary canonical representation.
  Then $k^*$ is called a \emph{auxiliary} end.
\end{definition}

The inner mapping $f$ induces a map $f^*$ on the set $B(S)$.
\begin{lemma}\label{lm:main_end_is_fixed_point}
  Let $k^*$ be a main end in $RIB$.
  Then $k^*$ is a fixed point of the map $f^*\colon B(S)\to B(S)$.
\end{lemma}

\begin{proof}
  Consider the main canonical
  backward chain $C_0=\bigcup_{i\le 0} A_i$ of the canonical
  representation of the end $k^*$.
  By definition of the main chain,
  $f(C_0)=\bigcup_{i\le 0} f(A_i)= f(A_0)\cup \bigcup_{i\le 1} A_{i-1})= f(A_0)\cup C_0$,
  and the canonical representation of the end $f^*(k^*)$ also
  contains the main canonical backward chain $C_0$.
  By Lemma~\ref{lm:main_canonical_sequence_to_chain},
  a main canonical backward chain is
  unique for its canonical representation, and canonical
  representation is unique for an end.
  Therefore. $f^*(k^*)=k^*$.
\end{proof}

\begin{corollary}\label{cor:auxiliary_end_is_preimage_of_fixed_point}
  Let $k^*$ be an auxiliary end in $RIB$.
  Then $k^*$ is a preimage of a main end in $RIB$ for some power of the map $f^*$.
\end{corollary}

\begin{lemma}\label{lm:RIB_no_ends_are_isolated}
  If $S$ has singular points, then no end in $RIB$ is isolated.
\end{lemma}

\begin{proof}
  Let $k^*$ be an end in $RIB$ with the canonical representation $(W_i)$.
  Suppose $(W_i)$ is main canonical representation (Definition~\ref{def:main_auxiliary_canonical_representation}).
  By Lemma~\ref{lm:main_canonical_sequence_to_chain},
  the canonical representation $(W_i)$ contains a unique main canonical
  backward chain $C_0=\cup_{i\le 0} A_i$.
  By Corollary~\ref{cor:backward_chain_split_homeo},
  there exists $N$ such that for all $n \ge N$,
  the atoms $A_{n}$ are of type $(*,a_{n})$, where $a_{n}>1$.
  Then, for all $n \le N$,
  the set $W_n$ contains another backward chain
  distinct from the subchain $C_n=\cup_{i\ge n} A_{i}$.
  Since the sequence $(\operatorname{Int}(W_n))$ forms a base of topology of point $k^*$ in $B(S)$,
  this implies that $k^*$ is not isolated in $B(S)$
  and, consequently, not isolated in $RIB$ as the sets $\operatorname{Int}(W_n)$ are
  neighborhoods of $k^*$ in $RIB$.

  Suppose $(W_i)$ is auxiliary canonical representation
  (Definition \ref{def:main_auxiliary_canonical_representation})
  and $C_0=\cup_{i\le 0} A_i$ is a pre-main auxiliary canonical backward
  chain of $(W_i)$ (Definition \ref{def:premain_auxiliary_canonical_backward_chain},
  Lemma \ref{lm:auxiliary_canonical_sequence_to_chain}).
  By definition, there exists $m>0$ such that $f^m(C_0)$
  is a main backward chain. 
  The map $f^m$ induces an action on $RIB$.
  It has already been shown that the end $f^m(k^*)$
  corresponding to the main backward chain
  $f^m(C_0)$  is not isolated in $RIB$. As a preimage of $f^m(k^*)$,
  $k^*$ is also not isolated in $RIB$.
  This completes the proof.
\end{proof}

\begin{lemma}\label{lm:RIB_is_closed}
  $RIB$ is a closed subset in $B(S)$.
\end{lemma}

\begin{proof}
  Let $(k^*_j)$ be a convergent sequence
  of ends belonging to $RIB$, with canonical representations $(W^j_i)$.
  Connected components of $(\omega_n)$ provide the base of topology in $RIB$.
  Therefore, for the sequence $(k^*_j)$ to be convergent,
  their canonical representations $(W^j_i)_{i>0}$ should stabilize,
  i.e. for all $i> 0$ there exists $J> 0$ such that for all $j>J$ $W^j_i$
  coincide. Denote this stabilized set as $W^{stabilized}_i$.
  The sequence $(W^{stabilized}_i)_{i>0}$ by construction is also a canonical
  representation of some end $k^*\in RIB$.
  Then $k^*$ is the limit point of $(k^*_j)$.
\end{proof}

\begin{lemma}\label{lm:RIB_ends_are_cantor}
  If $S$ has singular points, then
  $RIB$ is a Cantor set in $B(S)$.
\end{lemma}

\begin{proof}
  According to Lemmas~\ref{lm:RIB_no_ends_are_isolated}, \ref{lm:RIB_is_closed},
  $RIB$ is closed set and has no isolated points.
  Then $RIB$ is a perfect space.
  According to Lemma~\ref{lm:atoms_yuild_ends_infinity}, $RIB$ is zero-dimensional.
  Therefore, $RIB$ is a Cantor set.
\end{proof}

As a consequence of Lemmas~\ref{lm:0critpoint_component_topology},
\ref{lm:auxiliary_trunk_root_corellation},
\ref{lm:1crit_point_component_topology},
\ref{lm:AIB_ends_are_isolated},
\ref{lm:main_end_is_fixed_point},
\ref{cor:auxiliary_end_is_preimage_of_fixed_point},
\ref{lm:RIB_ends_are_cantor},
we obtain
\begin{theorem}\label{th:component_topology}
  $S$ is a planar set such that
  the ideal boundary of $S$ consists of two subsets $AIB$ and $RIB$.
  Both $AIB$ and $RIB$ are not empty.
  \begin{enumerate}[leftmargin=*]
  \item
    If $S$ has no singular points of $f$ then both
    $AIB$ and $RIB$ consist of single end;
    both those ends are fixed points for the induced map $f^*$ on $B(S)$.
  \item
    If $S$ has a single singular point of $f$ then
    $AIB$ consists of single end and $RIB$ is a Cantor set;
    all ends in $B(S)$ are fixed points for the induced map $f^*$ on $B(S)$.
  \item
    Otherwise $RIB$ is a Cantor set
    and $AIB$ consists 
    \begin{enumerate}[label={\rm\roman*)}]
    \item
      either of a single isolated end (punctured point),
      in which case
      all ends in $B(S)$ are fixed points for the induced map $f^*$ on $B(S)$,

    \item
      or of a countable set
      of isolated points. Then the main end in $AIB$ and the subset of main
      ends in $RIB$ are fixed points for the induced map $f^*$ on $B(S)$,
      and corresponding auxiliary ends are their preimages for some power of $f^*$.
\end{enumerate}
\end{enumerate}
\end{theorem}

\section{Note on general Pseudo-B\"{o}ttcher components}

\begin{definition}
  A connected component $K$ of a totally invariant set $R$ is called
  \begin{itemize}
  \item
  \emph{periodic}, if exists $n\ge 1$ such that $f^n(K)=K$;
  \item
  \emph{\preperiodic{}}, if exists $m\ge 1$ such that $f^m(K)$ is periodic;
  \item
  \emph{\forever wandering} otherwise.
\end{itemize}
\end{definition}

\begin{definition}\label{def:pseudobottcher_basin}
  A region of strict wandering $D$ is called
  \emph{Pseudo-B\"{o}ttcher} if for any connected component $S_i$ of $D$
  there exist $m\ge 0$ and $n\ge 1$ such that the set $f^m(S_i)$ is
  totally invariant Pseudo-B\"{o}ttcher component
  with respect to the map $f^n$.
  Connected components $S_i$ of $B$ are called \emph{Pseudo-B\"{o}ttcher}.
\end{definition}

By construction, $D$ is totally invariant set.
However, by Definition~\ref{def:pseudobottcher_basin},
Pseudo-B\"{o}ttcher region of strict wandering $D$ can't have
wandering components.
Then any connected component of $D$ is either periodic or \preperiodic{}.
$f$ eventually maps \preperiodic{} ones onto periodic ones.
Any periodic component $S$ of period $k$ is invariant for $f^k$.
Moreover, $S$ is totally invariant for the restriction $f^k|_S$.
The fact that the ends of $S$'s ideal boundary belong to the attractor
and the repeller also holds true for any power of $f$.
Therefore the statements~\ref{lm:0critpoint_component_topology}
and \ref{th:component_topology}
also hold true for any $S$ of $D$.

Consider a \preperiodic{} connected component $S$ of $D$ such that
$\exists l\ge 1$ $f^l(S)$ is periodic.
$f^l$ is a branched covering of $M$.
According to Theorem~\ref{th:component_topology}, $f^l(S)$ does not
have handles.
Using the Riemann-Hurvitz formula, we obtain that $S$ does not have handles too.
However, the Riemann-Hurvitz formula does not prohibit the case when
$S$ has more ends then $f^l(S)$.

For example, if the topological type of $f^l(S)$ is a sphere with 2
punctured points, then the topological type of $S$ can be a sphere with
3 punctured points, because that kind of branched covering is allowed.

\bibliographystyle{proc_igc_plain}
\bibliography{Vlasenko-article12-en}


\end{document}